\documentclass{amsart}
\usepackage{amssymb} 
\usepackage{amsmath} 
\usepackage{mathrsfs} 
\usepackage{eufrak}
\usepackage{amscd}
\usepackage{amsbsy}
\usepackage{comment}
\usepackage{hyperref}
\usepackage{mathtools}
\usepackage[matrix,arrow]{xy}

\usepackage{esint}
\usepackage[letterpaper, margin=1.1in]{geometry}
\usepackage{fancyhdr}
\usepackage[shortlabels]{enumitem}
\usepackage{xcolor}

\newtheorem{thm}{Theorem}[section]
\newtheorem{lem}[thm]{Lemma}
\newtheorem{prop}[thm]{Proposition}
\newtheorem{cor}[thm]{Corollary}
\newtheorem{remark}[thm]{Remark}

\theoremstyle{definition}
\newtheorem{defn}{Definition}[section]

\begin{document}

\setlength\parindent{0pt}

\vspace{-3em}

        \author{Albert Chau$^1$}
        \address{Department of Mathematics,
                The University of British Columbia, Room 121, 1984 Mathematics
                Road, Vancouver, B.C., Canada V6T 1Z2} \email{chau@math.ubc.ca}

        \thanks{$^1$Research
                partially supported by NSERC grant no. \#327637-06}

        \author[A. Martens]{Adam Martens}
        \address{Department of Mathematics, The University of British Columbia,
1984 Mathematics Road, Vancouver, B.C.,  Canada V6T 1Z2}
\email{martens@math.ubc.ca}

        \bibliographystyle{amsplain}
\title{Long-time Ricci flow existence and topological rigidity from manifolds with pinched scale-invariant integral curvature}

\maketitle

\pagestyle{plain}

\vspace{-3em}

\begin{abstract}
We prove long-time existence of the Ricci flow starting from complete manifolds with bounded curvature and scale-invariant integral curvature sufficiently pinched with respect to the inverse of its Sobolev constant. Moreover, if the curvature is sub-critical $L^p$-integrable, this flow converges locally smoothly to a limiting metric $g(\infty)$ on $M$ with $(M,g(\infty))$ isometric to the standard flat $\mathbb{R}^n$, which implies topological rigidity of $M$. This generalizes work of Chen \cite{ChenEric}, who proved analogous results for asymptotically flat manifolds. We also prove a long-time Ricci flow existence (and likewise topological rigidity) result for unbounded curvature initial data, assuming the initial data is a locally smooth limit of bounded curvature manifolds as described above.
\vspace{-1em}
\end{abstract}

\begin{center}
\vspace*{2em}
{\itshape \textbf{In celebration and memory of Richard S. Hamilton}}
\vspace*{2em}
\end{center}

\maketitle
\tableofcontents

\section{Introduction}
The Ricci flow is the following evolution equation for an initial Riemannian metric $g_0$ on a $n$-dimensional manifold $M^n$:
\begin{equation}\label{RF}
                \left\{
                \begin{aligned}
                        \frac{{\partial }}{{\partial t}}g(t) &= -2 \text{Rc}_{g(t)}, \\
                        g(0) &= g_0.
                \end{aligned}
                \right.
        \end{equation}

A classical result of Shi \cite{Shi} says that if $g_0$ is complete with sectional curvatures uniformly bounded by $K$, then \eqref{RF} admits a complete bounded curvature solution $g(t)$ on $M^n\times[0, T(n, K))$ for some time $T(n, K)>0$ depending only on $n$ and $K$.  Extending Shi's theorem to cases of unbounded curvature, or by extending the solution for all times $t\in [0, \infty)$ under suitable conditions are of fundamental importance in Ricci flow theory and applications.  One theme of recent prevalence is to consider manifolds satisfying some Sobolev-type control coupled with a smallness of the scale-invariant integral curvature 
\[
\int_\Omega |\text{Rm}|^{\frac{n}{2}}\,dV_g,
\]
which is often referred to as the \textit{curvature concentration} (this smallness is thought of as a type of almost Euclidean condition). By \emph{Sobolev-type control} we typically mean the validity of an $L^2$-Sobolev inequality (see \eqref{mainthme1a}). Such inequalities play a fundamental role in the analysis of Riemannian manifolds and partial differential equations and, in particular, imply a Faber--Krahn inequality. In our setting, an $L^2$-Sobolev inequality is equivalent to boundedness of the $\bar\nu$- or $\nu$-functionals (see \cite[Section~3]{Martens}), and we will therefore also use \emph{Sobolev-type control} to refer to either of these conditions. See for instance \cite{CCL} or \cite{CHL} where short time existence of \eqref{RF} is proved assuming local Sobolev-type control and local smallness of the curvature concentration (see also \cite{Martens, MartensScalar, LeeTam} which were all posted after the initial submission of this manuscript). If one replaces the local assumptions in exchange for global ones, one may obtain stronger conclusions. For example, in \cite{ChanLee} a gap theorem is proved for manifolds which have nonnegative Ricci curvature, Euclidean volume growth, and globally small scale-invariant integral curvature as a consequence of rescalings of the short-time Ricci flow existence results proved in \cite{CHL}. Another example is the following rigidity theorem of Chen for asymptotically flat manifolds.


\begin{thm}{\cite[Theorem A]{ChenEric}}\label{Chenmain}
For $n\geq 3$, there exists a dimensional constant $\delta(n)>0$ such that the following holds. Let $(M^n, g)$ be an asymptotically flat manifold of order $\tau>0$. Suppose that $(M^n,g)$ satisfies the Sobolev inequality
\[
\left(\int u^{\frac{2n}{n-2}}\,dV_g\right)^{\frac{n-2}{n}}\leq C_g \int |\nabla u|^2\,dV_g
\]
for all $u\in W^{1,2}(M, g)$, and the scale-invariant curvature pinching
\[
\left(\int_{M^n}|\text{Rm}|_{g}^{\frac{n}{2}}\,dV_{g}\right)^{\frac{2}{n}}\leq \delta(n) \frac{1}{C_{g}}.
\] 
Then the complete Ricci flow $g(t)$ with initial condition $g(0)=g$ exists for all times $t\in [0,\infty)$ and satisfies
\[
\lim_{t\to \infty} t\sup_{x\in M} |\text{Rm}|(x,t) =0.
\]
In particular, $g(t)$ converges in $C^\infty_{-\tau'}(M^n)$ for any $\tau'\in (0,\min(\tau, n-2))$ to the $\mathbb{R}^n$ with the Euclidean metric and thus $M^n$ is diffeomorphic to $\mathbb{R}^n$. 
\end{thm}


The purpose of this paper is to generalize Theorem \ref{Chenmain} to complete noncompact bounded curvature manifolds. We prove the following. 


\begin{thm}\label{main}
For $n\geq 3$, there are dimensional constants $\delta(n)>0$ and $C_0(n)=(5e)^{\frac{n}{2}}$ such that the following holds. Let $(M^n,g)$ be a complete bounded curvature manifold. Suppose that $(M^n,g)$ satisfies the Sobolev inequality
\begin{equation}\label{mainthme1a}
\left(\int u^{\frac{2n}{n-2}}\,dV_g\right)^{\frac{n-2}{n}}\leq C_g \int |\nabla u|^2\,dV_g
\end{equation}
for all $u\in W^{1,2}(M, g)$, and the scale-invariant curvature pinching
\begin{equation}\label{mainthme1}
\left(\int_M |\text{Rm}|_g^{\frac{n}{2}}\,dV_g\right)^{\frac{2}{n}}\leq \delta(n) \frac{1}{C_{g}}.
\end{equation}
Then the complete Ricci flow $g(t)$ with initial condition $g(0)=g$ exists for all times $t\in [0,\infty)$ and satisfies the following curvature estimates:
\begin{enumerate}[(a)]
\item \label{maincon2} For all $t\in (0,\infty)$ there holds
\[
\sup_{x\in M}|\text{Rm}|(x,t)\leq \frac{C_0(n)}{t}.
\]
\item \label{maincon3} The singularity at infinity satisfies
\[
\lim_{t\to \infty} t \sup_{x\in M} |\text{Rm}|(x,t) =0.
\]
\end{enumerate}
Moreover, if $|\text{Rm}|_g\in L^p$ for some $p<\frac{n}{2}$, then $g(t)$ converges locally smoothly to a complete flat limit metric $g(\infty)$ on $M^n$ with Euclidean volume growth, and thus $M^n$ is diffeomorphic to $\mathbb{R}^n$. 
\end{thm}

\begin{remark}
    Since the initial submission of this manuscript, S. Huang and Z. Peng have removed the requirement of $|\text{Rm}|_g\in L^p$ in order to obtain the diffeomorphism in Theorem \ref{main}, see {\cite[Corollary 1.3]{HuangPeng}}. 
\end{remark}


\begin{remark}
The constant $\delta(n)$ in Theorem \ref{main} can be taken to be  
\[
\delta(n)=\frac{1}{40nL(n)},
\]
where the dimensional constant $L(n)$ is from Theorem \ref{C} and is estimated in the appendix. This value of $\delta(n)$ is henceforth fixed in the remainder of the paper.
\end{remark}

\begin{remark}
When $(M^n,g)$ is asymptotically flat, the condition that $|\text{Rm}|\in L^p$ for some $p<\frac{n}{2}$ is automatically satisfied whenever $p>\frac{n}{2+\tau}$ where $\tau>0$ is the order of the asymptotic end. Our proof of Proposition \ref{conv4.10} establishes that for any $\sigma<\tau/2$, we have the global decay estimate $|\nabla^k \text{Rm}|\lesssim t^{-1-\sigma- k/2 }$. Using this, it follows from {\cite[Theorem 5.1]{Li}} that in fact $g(t)$ converges in $C^\infty_{-\tau'}(M^n)$ for any $\tau'\in (0,\min(\tau, n-2))$ to $\mathbb{R}^n$ with the Euclidean metric, in the same sense as Theorem \ref{main}. Thus, Theorem \ref{main} generalizes Theorem \ref{Chenmain}.
\end{remark}


We also prove - as a corollary of our main theorem - the following long-time Ricci flow existence and rigidity result for which the initial condition may have unbounded curvature. 


\begin{thm}\label{main2}
With $n\geq 3$ and $\delta(n)$, $C_0(n)$ as in Theorem \ref{main}, the following holds. Suppose that $(M^n,g)$ is a complete Riemannian manifold (not necessarily bounded curvature) which is the pointed smooth local limit of a sequence of complete bounded curvature Riemannian manifolds $(M_i, g_i)$, that is to say 
\[
\phi_i^* g_i \xrightarrow{C^{\infty}_{loc}(M)}  g
\]
for some functions $\phi_i:\Omega_i\to M_i$ where each $\phi_i$ is a diffeomorphism onto its image and the $\Omega_i$ form an exhaustion of $M$ by nested compact sets. Moreover, assume that the following hold for each $i\in \mathbb{N}$:
\begin{enumerate}[(i)]
\item \label{main2i} The Sobolev inequality
\[
\left(\int u^{\frac{2n}{n-2}}\,dV_{g_i}\right)^{\frac{n-2}{n}}\leq C_{g_i} \int |\nabla u|^2\,dV_{g_i}
\]
for all $u\in W^{1,2}(M, g_i)$.
\item \label{main2ii} Sufficiently pinched scale-invariant integral curvature conditions
\[
\sup_{i\in \mathbb{N}}\left(\int_M |\text{Rm}|_{g_i}^{\frac{n}{2}}\,dV_{g_i}\right)^{\frac{2}{n}}\leq \frac{\delta(n)}{\sup_{i\in \mathbb{N}}C_{g_i}}.
\]
\end{enumerate}
Then there exists a complete solution $g(t)$ to the Ricci flow on $M\times [0,\infty)$ with $g(0)=g$ that satisfies the curvature estimate
\[
\sup_{x\in M}|\text{Rm}|(x,t)\leq \frac{C_0(n)}{t}
\]
for all $t\in (0,\infty)$. Moreover if $\sup_{i\in \mathbb{N}}\||\text{Rm}|_{g_i}\|_{L^p(g_i)}<\infty$ for some $p<\frac{n}{2}$,  then $g(t)$ converges locally smoothly to a complete flat limit metric $g(\infty)$ on $M^n$ with Euclidean volume growth, and thus $M^n$ is diffeomorphic to $\mathbb{R}^n$. 
\end{thm}


\begin{remark} 
The hypotheses of Theorem \ref{main} impose no restriction on the supremum of, or decay/growth rate of the curvature. In fact, given \textbf{any} sequence of positive numbers $\{\alpha_i\}\subset \mathbb{R}_+$, one can construct a complete metric $g$ on $\mathbb{R}^n$ such that 
\[
\sup_{ B_{g_E}(i+1)\setminus B_{g_E}(i)} |\text{Rm}|_g =\alpha_i
\]
but for which $g$ satisfies the Sobolev inequality \eqref{mainthme1a} and the integral curvature pinching estimate \eqref{mainthme1} with $C_g$ arbitrarily close to $C_{g_E}$. In particular, one can define a sequence of complete bounded curvature metrics $g_i$ on $\mathbb{R}^n$ that converge locally smoothly to an unbounded curvature metric $g$ on $\mathbb{R}^n$ such that the metrics $g_i$ satisfy the hypotheses \ref{main2i} and \ref{main2ii} from Theorem \ref{main2}. The described constructions can be accomplished by gluing into the flat Euclidean space $(\mathbb{R}^n, g_E)$ countably many ``wrinkled balls", by which we mean a compact perturbation $h$ of the Euclidean metric on the open unit ball $B$ which satisfies: (a) $(1+c_1)^{-1}g_E\leq h\leq (1+c_1)g_E$ on $B$, (b) $\sup_B |\text{Rm}|_{h}=c_2$, and (c) $\int_B|\text{Rm}|_{h}^{\frac{n}{2}}\,dV_{h}<c_3$, where $c_1,c_2,c_3>0$ can be prescribed arbitrarily. 
\end{remark}


The organization of this paper is as follows. Section \ref{secboundedn/2} is devoted to showing that finiteness of $\int |\text{Rm}|^p\,dV_t$ ($p \geq 1$) is preserved along complete bounded curvature Ricci flows. Though this is a standard maximum principle argument, the conclusion is essential to the proof of Theorem \ref{main}. In Section~\ref{sobalongflow} we discuss why we have a uniform global Sobolev inequality along the flow, which follows from work of Ye \cite{Ye} (see also \cite{Zhang}). Section \ref{proofofmain} contains the bulk of the proof of Theorem \ref{main} following the outline of Chen \cite{ChenEric} very closely. In fact, the numbering of the theorems in that section are specifically chosen to be the same as the analogous results in \cite{ChenEric}. The proof of Theorem~\ref{main2} is carried out in Section \ref{secproofmain2}. Finally in the appendix we obtain a bound on the constant $\delta(n)$ by tracking through the constants from \cite{Ye}.


\section*{Acknowledgement}

The authors would like to thank Eric Chen and the referees for the helpful comments on this manuscript, as well as Oliver D\'iaz Espinosa for pointing out to us an improved bound in the Marcinkiewicz Interpolation Theorem~\ref{Mark}.


\section{Boundedness of $\int |\operatorname{Rm}|^p\,dV_t$ along the Ricci flow}\label{secboundedn/2}


For this section we fix the notation that $(M,g(t))$, $t\in [0,T]$ is a complete bounded curvature Ricci flow and we write
\[
K=\sup_{M\times [0,T]}|\text{Rm}|(x,t) <\infty.
\]
Also, given a function $f(x, t)$ on $M^n\times[0, T]$ we will
often abbreviate as
\[
\int f dV_t:= \int_M f(x, t) dV_{g(t)} (x).
\]

We will prove that $\int |\text{Rm}|^p\,dV_t$ remains finite for all $t\in [0,T]$ and any $p\geq 1$, assuming it was finite initially. In fact we show that this quantity is upper-continuous, which will be used in the proof of our main theorem. If we were to restrict to $p\geq 2$, then the results in this section would be much simpler. When $1\leq p<2$ however (for example if $n=3$ and $p=\frac{n}{2}$), there is an added difficulty due to the fact that 
\[
\partial_t |\text{Rm}|^{p}=\frac{p}{2} |\text{Rm}|^{p-2}\partial_t |\text{Rm}|^2,
\]
and the curvature is certainly permitted to be zero in places.\\

We get around this as follows. First, choose a smooth positive function $v_0: M\to \mathbb{R}$ which is integrable (with respect to $g_0$) and satisfies $0<v_0\leq 1$. Now for any $\epsilon>0$, define $v_\epsilon$ to be the solution to the initial value problem 
\[
\begin{cases}
&\partial_t v_\epsilon =\Delta_{g(t)}v_\epsilon\\
& v_\epsilon(0,x)=\epsilon v_0(x).
\end{cases}
\]
Coupling the well-known evolution 
\[
\partial_t |\text{Rm}|^2\leq \Delta |\text{Rm}|^2-2|\nabla\text{Rm}|^2+16|\text{Rm}|^3
\]
with the evolution of $v_\epsilon^2$ which satisfies
\[
\partial_t v_\epsilon^2 = \Delta v_\epsilon^2 - 2|\nabla v_\epsilon|^2,
\]
we can write
\[
\partial_t (|\text{Rm}|^2+v_\epsilon^2 )\leq \Delta (|\text{Rm}|^2+v_\epsilon^2)-2|\nabla\text{Rm}|^2- 2|\nabla v_\epsilon|^2+16|\text{Rm}|^3.
\]
By Cauchy-Schwarz and Young's inequality, there holds 
\begin{align*}
|\nabla (|\text{Rm}|^2+v_\epsilon^2)|^2&=4|\text{Rm}|^2 |\nabla |\text{Rm}||^2+4 v_\epsilon^2 |\nabla v_\epsilon|^2+8 |\text{Rm}| v_\epsilon \langle \nabla |\text{Rm}|, \nabla v_\epsilon \rangle \\&\leq 
4|\text{Rm}|^2 |\nabla |\text{Rm}||^2+4 v_\epsilon^2 |\nabla v_\epsilon|^2+8 |\text{Rm}| v_\epsilon  |\nabla |\text{Rm}|| \,|\nabla v_\epsilon |\\&\leq 
4|\text{Rm}|^2 |\nabla |\text{Rm}||^2+4 v_\epsilon^2 |\nabla v_\epsilon|^2+4 |\text{Rm}|^2|\nabla v_\epsilon|^2 \\
																		  &\quad +v_\epsilon^2|\nabla |\text{Rm}||^2 \\&=
4(|\text{Rm}|^2+v_\epsilon^2)(|\nabla |\text{Rm}||^2+|\nabla v_\epsilon|^2).
\end{align*}
So if we set $K_\epsilon^2=|\text{Rm}|^2+v_\epsilon^2$, we can see that
\[
4K_\epsilon^2 |\nabla K_\epsilon|^2\leq 4K_\epsilon^2 (|\nabla |\text{Rm}||^2+|\nabla v_\epsilon|^2).
\]
Subsequently, by the Kato inequality $|\nabla |T||\leq |\nabla T|$, there holds
\[
\partial_t K_\epsilon^2\leq \Delta K_\epsilon^2-2|\nabla K_\epsilon|^2+16K_\epsilon^3.
\]

We now use this evolution of $K_\epsilon^2$ to determine a useful evolution inequality for $K_\epsilon^p$. 


\begin{lem}
For any $p\geq 1$, there holds 
\begin{equation}\label{evolutionofn/2}
\partial_t K_\epsilon^{p}\leq \Delta K_\epsilon^{p}+8p K_\epsilon^{p+1}.
\end{equation}
\end{lem}

\begin{proof}
Denote $\alpha=\frac{p}{2}$. Since $K_\epsilon$ is a strictly positive function on $M\times [0,T]$, we freely differentiate and see that
\begin{align*}
\partial_t K_\epsilon^{2\alpha}&=\partial_t (K_\epsilon^2)^{\alpha}=\alpha K_\epsilon^{2(\alpha-1)}\partial_t K_\epsilon^2\\&\leq 
\alpha K_\epsilon^{2(\alpha-1)}(\Delta K_\epsilon^2-2|\nabla K_\epsilon|^2+16K_\epsilon^3)\\&=
\alpha K_\epsilon^{2(\alpha-1)}\Delta K_\epsilon^2-2\alpha K_\epsilon^{2(\alpha-1)}|\nabla K_\epsilon|^2+16\alpha K_\epsilon^{2\alpha+1}\\&=
\Delta K_\epsilon^{2\alpha}-4\alpha(\alpha-1)K_\epsilon^{2(\alpha-1)}|\nabla K_\epsilon|^2-2\alpha K_\epsilon^{2(\alpha-1)}|\nabla K_\epsilon|^2\\
																	  &\quad +16\alpha K_\epsilon^{2\alpha+1}\\&\leq 
\Delta K_\epsilon^{2\alpha}+16\alpha K_\epsilon^{2\alpha+1},
\end{align*}
where in the last line we have used the fact that $\alpha\geq \frac{1}{2}$ implies that $-4\alpha(\alpha-1)-2\alpha\leq 0$.
\end{proof}


We can use the heat-type evolution of $K_\epsilon^p$ as in \eqref{evolutionofn/2}, along with the Maximum Principle for complete bounded curvature Ricci flows to obtain the following. 


\begin{prop}\label{finitenessofn/2}
Let $p\geq 1$ and suppose that $|\text{Rm}|_{g(0)}\in L^p_{g(0)}$. Then there holds
\[
\int_M |\text{Rm}|_{g(t)}^p\,dV_t\leq e^{(8p+n)Kt}\int_M |\text{Rm}|_{g(0)}^p\,dV_0.
\]
\end{prop}


\begin{proof}
First let $\Omega_0 \subset \subset \Omega_1 \subset \subset \dots \subset \subset M$ be an exhausting sequence of relatively compact domains with smooth boundaries in $M$, and for each $j\in \mathbb{N}$ let $\phi_j:M\to [0,1]$ be a smooth cutoff function which is supported in $\Omega_{j}$ and $\phi_j\equiv 1$ in $\Omega_{j-1}$. Now for any $\epsilon>0$ and for each $j\in \mathbb{N}$, we denote by $u_{j}:M\times [0,T]\to \mathbb{R}$ the solution to the initial value Dirichlet boundary condition problem
\[
\begin{cases}
\partial_t u_{j}=\Delta_{g(t)}u_{j}\\
u_{j}(x,0)=\phi_j K_\epsilon^p(x,0)\\
u_{j}(x,t)=0 \text{ if } x\in \partial \Omega_j. 
\end{cases}
\]
Then by the Maximum Principle, we know that the $u_{j}$ are all nonnegative functions which are uniformly bounded 
\begin{equation}\label{unifbound2}
\sup_{M\times [0,T]}u_{j}(x,t)\leq \sup_{x\in M}K_\epsilon^p(x,0)\leq \sup_{x\in M} (|\text{Rm}|^2(x,0)+\epsilon^2)^{\frac{p}{2}}.
\end{equation}
Then by the Dominated Convergence Theorem and the evolution $\partial_t dV_t=-R_t dV_t$, we have
\begin{align}\label{prop2.2e1}
	\frac{d}{dt}\int_{\Omega_j}u_{j}\,dV_t&=\int_{\Omega_j}\partial_t u_{j}-R_t u_{j}\,dV_t\\
					     \notag  &=\int_{\Omega_j}\Delta_{g(t)}u_{j}-R_t u_j\,dV_t \\
					     \notag & \leq \int_{\Omega_j}\Delta_{g(t)}u_{j}+nK\int_{\Omega_j} u_{j}\,dV_t.
\end{align}
Here we have applied the bound $|R|\leq n|\text{Rm}|$, see {\cite[Proposition 7.28]{Lee}} for instance. Now Green's Theorem says that 
\[
\int_{\Omega_j}\Delta_{g(t)}u_{j}\,dV_t=\int_{\partial\Omega_j} \nabla_{v(t)} u_{j}\,dV_t
\]
where $v(t)$ is the outward unit normal on $\partial\Omega_j$ relative to $g(t)$. Then by the nonnegativitiy of $u_j$ and the Dirichlet condition $u_j=0$ on $\partial \Omega_j$, this boundary integral is nonpositive. Subsequently, putting this back into \eqref{prop2.2e1} shows that 
\begin{align}\label{prop2.2intbound}
	\int_M u_j\,dV_t=\int_{\Omega_j}u_j\,dV_t &\leq e^{nKt}\int_M u_j(x,0)\,dV_0\\
	\notag &	\leq e^{nKt}\int_M (|\text{Rm}|_{g(0)}^2+\epsilon^2 v_0^2)^{\frac{p}{2}}\,dV_0.
\end{align}

Moreover for each fixed point in spacetime $(x,t)\in M\times [0,T]$, the sequence $u_j(x,t)$ is monotone increasing (another consequence of the Maximum Principle), and thus has a pointwise limit which we denote by $u(x,t)$. By the uniform boundedness \eqref{unifbound2}, Schauder estimates, the Arzel\`a-Ascoli Theorem, and a diagonal subsequence argument, we obtain that the limit function $u$ is smooth on $M\times [0,T)$ and solves the initial value problem
\[
\begin{cases}
\partial_t u=\Delta_{g(t)}u\\
u(x,0)=K_\epsilon^p(x,0).
\end{cases}
\]
Moreover, the integral bound \eqref{prop2.2intbound} descends to $u$ as well via the Monotone Convergence Theorem. That is to say
\begin{equation}\label{uintbound2.2}
\int_M u\,dV_t\leq e^{nKt}\int (|\text{Rm}|_{g(0)}^2+\epsilon^2 v_0^2)^{\frac{p}{2}}\,dV_0.
\end{equation}
We also define $\tilde u(x,t)=e^{8p t(K^2+\epsilon^2)^{\frac{1}{2}}}u(x,t)$ so that
\begin{align*}
	\partial_t \tilde u&=8p (K^2+\epsilon^2)^{\frac{1}{2}}\tilde  u+e^{8p t(K^2+\epsilon^2)^{\frac{1}{2}}}\partial_t  u\\
			   &=8p (K^2+\epsilon^2)^{\frac{1}{2}}\tilde u+e^{8p t(K^2+\epsilon^2)^{\frac{1}{2}}}\Delta_{g(t)} u\\
			   &=8p (K^2+\epsilon^2)^{\frac{1}{2}}\tilde u +\Delta_{g(t)}\tilde u,
\end{align*}
and $\tilde u(x,0)=K_\epsilon^p(x,0)$. Comparing this with the evolution inequality \eqref{evolutionofn/2}, another application of the Maximum Principle shows that 
\[
K_\epsilon^p(x,t)\leq \tilde u(x,t)
\]
for all $(x,t)\in M\times [0,T]$. Integrating this and applying our integral bound~\eqref{uintbound2.2} yields
\begin{align*}
\int_M |\text{Rm}|_{g(t)}^p\,dV_t&\leq \int_M K_\epsilon^p\,dV_t\leq \int_M \tilde u(x,t)\,dV_t(x)\\&=e^{8p t(K^2+\epsilon^2)^{\frac{1}{2}}}\int_M u(x,t)\,dV_t(x)\\&\leq e^{(8p+n) t(K^2+\epsilon^2)^{\frac{1}{2}}}\int_M (|\text{Rm}|_{g(0)}^2+\epsilon^2 v_0^2)^{\frac{p}{2}}\,dV_0.
\end{align*}
Now because of the assumed integrability of $|\text{Rm}|_{g(0)}^p$ and $v_0$ (with respect to $g_0$), we can take $\epsilon\to 0$ and once again apply the Dominated Convergence Theorem to obtain the result. 
\end{proof}


\section{Sobolev Inequality along the Ricci flow}\label{sobalongflow}


The purpose of this section is to show that we have a uniform global Sobolev inequality along a complete bounded curvature Ricci flow (Theorem \ref{C}). This has been proved on compact manifolds in \cite{Ye} and this result was verified and applied in various complete noncompact cases such as \cite{ChenEric} in the asymptotically flat case. The steps to proving this are as follows: (1) Show a log-Sobolev inequality for the initial metric $g_0$; (2) Show the log-Sobolev inequality is preserved along the flow using monotonicity of Perelman's $\mu$-functional; and (3) At a given time $t\in [0,T]$ along the flow, use regularity properties of the heat equation on $M$ with the \emph{fixed} metric $g(t)$ to get back to a Sobolev inequality from the log-Sobolev inequality at time $t$. We briefly outline these steps and refer the reader to the other works in which these methods have been employed for more detail.

\medskip

\noindent \textbf{Step 1: log-Sobolev inequality for the initial metric $g_0$.}

\vspace{5pt}

We first prove a log-Sobolev inequality for $g_0$. This has been done many times in the literature (cf. \cite{Ye}, \cite{ChenEric}, etc.), but we will sketch the proof here for convenience and for our future reference. 


\begin{lem}[Log-Sobolev inequality for $g_0$]\label{logsobinitial}
Let $(M^n, g)$, $n\geq 3$ be any Riemannian manifold satisfying 
\begin{equation}\label{logsobinitial1}
\left(\int |u|^{\frac{2n}{n-2}}\,dV_g\right)^{\frac{n-2}{n}}\leq C_g \int |\nabla u|^2\,dV_g
\end{equation}
for all $u\in W^{1,2}(M,g)$. Then each of the following is true: \\

\noindent \textbf{(1)} For all $u\in W^{1,2}(M,g)$ with $\int u^2\,dV_g=1$ and any $\tau>0$ there holds
\[
\int u^2\log u^2\,dV_g \leq 4\tau \int |\nabla u|^2\,dV_g -\frac{n}{2}\log \tau +\frac{n}{2}\left(\log C_g+\log \frac{n}{8}-1\right).
\]
\textbf{(2)} Assume moreover that 
\[
\left(\int |R_g|^{\frac{n}{2}}\,dV_g\right)^{\frac{2}{n}}\leq \frac{2}{C_g}.
\]
Then for all $u\in W^{1,2}(M,g)$ with $\int u^2\,dV_g=1$ and any $\tau>0$ there holds
\begin{align}\label{l3.1part2}
	\int u^2\log u^2\,dV_g &\leq 4\tau \left(\int |\nabla u|^2+\frac{R_g}{4}u^2\,dV_g\right) \\
			     \notag   &\quad -\frac{n}{2}\log \tau +\frac{n}{2}\left(\log C_g+\log \frac{n}{4}-1\right).
\end{align}
\end{lem}


\begin{proof}
We denote $q=\frac{2n}{n-2}$ and apply Jensen's inequality and \eqref{logsobinitial1} to see
\begin{align*}
\int_M u^2 \log u^2 \,dV&=
2\int_M u^2 \log u \,dV=
\frac{2}{q-2}\int_M u^2 \log u^{q-2}\,dV\\&\leq 
\frac{2}{q-2}\log\int_M u^q\,dV=
\frac{q}{q-2}\log \|u\|_q^2 \\
					  &\leq 
\frac{n}{2}\log \left(C_g\int |\nabla u|^2\,dV_g\right).
\end{align*}
Then \textbf{(1)} is established after using the inequality
\[
\log \left(\int |\nabla u|^2\,dV_g\right)\leq a\int |\nabla u|^2\,dV_g-1-\log a
\]
for $a=\frac{8\tau}{n}$. To prove \textbf{(2)} we use H\"older's inequality and \eqref{logsobinitial1} to see that 
\begin{align*}
	2\tau \int R_g u^2\,dV_g &\geq
-2\tau \left(\int |R_g|^{\frac{n}{2}}\,dV_g\right)^{\frac{2}{n}}\left(\int |u|^{\frac{2n}{n-2}}\,dV_g\right)^{\frac{n-2}{n}}\\
				 &\geq
-4\tau \int |\nabla u|^2\,dV_g.
\end{align*}
Then plugging this into  \textbf{(1)} gives
\begin{align*}
	\int u^2\log u^2\,dV_g &\leq 8\tau \int |\nabla u|^2\,dV_g +2\tau \int R_g u^2\,dV_g-\frac{n}{2}\log \tau \\
			       &\quad +\frac{n}{2}\left(\log C_g+\log \frac{n}{8}-1\right).
\end{align*}
This is equivalent to \eqref{l3.1part2} after relabelling. 
\end{proof}


\noindent\textbf{Step 2: log-Sobolev inequality is preserved along the flow.}
\vspace{5pt}

The preservation of the log-Sobolev inequality obtained in Lemma \ref{logsobinitial} for $g=g(0)$, to times $t>0$ along the Ricci flow is shown via the monotonicity of Perelman's $\mu$-functional. Let's recall the definition of this. 


\begin{defn}[Perelman's $\mu$-functional]
Given a complete bounded curvature Riemannian manifold $(M^n,g)$, $n\geq 3$ and $\tau>0$, we define the $\mathcal{W}$-functional to be 
\[
\mathcal{W}(g,u,\tau)=\int \left[ \tau(4|\nabla u|^2+R u^2)-u^2 \log u^2\right]\,dV_g -n-\frac{n}{2}\log (4\pi \tau),
\]
where $u\in W^{1,2}(M,g)$ satisfies $\int u^2\,dV_g=1$. Then \emph{Perelman's $\mu$-functional} is given by
\[
\mu(g,\tau)=\inf\left\{\mathcal{W}(g,u,\tau) \; : \; u\in W^{1,2}(M,g), \;\; \int u^2\,dV_g=1\right\}.
\]
\end{defn}


That the $\mu$-functional is a monotone quantity along the Ricci flow was first shown by Perelman \cite{Per} in the compact case, and has since received a lot of attention for the noncompact case (cf. \cite{CTY}, \cite{Zhang}, \cite{ChenEric}, etc.). The following lemma is the precise statement we will require, and follows from the (much more general) work in \cite{WangA}. 


\begin{lem}\label{monotonemufunct}
Suppose that $(M^n,g(t))$, $t\in [0,T]$ is a complete bounded curvature Ricci flow. Then for any $t\in [0,T]$ and $\tau>t$, there holds
\[
\mu(g(0), \tau)\leq \mu(g(t), \tau-t).
\]
\end{lem}


\begin{proof}[Sketch of Proof]
    Fix $t\in [0,T]$ and $\tau>t$. For all large enough $A$ (depending on the bound of the curvature along the flow and $\tau$), we apply the almost-monotonicity of the local $\mu$-functional {\cite[Theorem 5.4]{WangA}} to see that 
    \[
    \mu\big(B_{g(t)}(x_0,8A\sqrt{t}),g(t),\tau - t\big)-\mu\big(B_{g(0)}(x_0,20A\sqrt{t}),g(0),\tau \big)\geq -A^{-2}.
    \]
    The result follows by the lower-continuity of the local $\mu$-functional {\cite[Proposition 2.3]{WangA}} when we take $A\to \infty$.
\end{proof}

So if we have any complete bounded curvature Ricci flow $(M,g(t))$, $t\in [0,T]$ such that inequality \eqref{l3.1part2} holds for all $L^2$-normalized $W^{1,2}(M, g(t))$ functions $u$, then this establishes the inequality
\[
\mu(g(0),\tau+t) \geq -\frac{n}{2}\left(\log C_{g(0)}+\log \frac{n}{4}-1\right) -n-\frac{n}{4}\log (4\pi)
\]
for any $t\in [0,T]$ and any $\tau>0$. Now the monotonicity of the $\mu$-functional implies that the above holds with $\mu(g(t), \tau)$ replacing $\mu(g(0), \tau+t)$. This in turn is equivalent to a log-Sobolev inequality at time $t$. More precisely, there holds 
\begin{align}\label{sobineqalongflowmufunct}
	\int u^2\log u^2\,dV_{g(t)} &\leq 4\tau \left(\int |\nabla u|_{g(t)}^2+\frac{R_{g(t)}}{4}u^2\,dV_{g(t)}\right) -\frac{n}{2}\log \tau \\
				 \notag    &\quad +\frac{n}{2}\left(\log C_{g(0)}+\log \frac{n}{4}-1\right)
\end{align}
for any $\tau>0$. 

\medskip

\noindent \textbf{Step 3: log-Sobolev inequality implies Sobolev inequality.}

\vspace{5pt}

The method of obtaining a Sobolev inequality from a log-Sobolev inequality is well known in the literature: Davies {\cite[Chapters 1-2]{Davies}} describes the process in the general context of abstract measure spaces $(\Omega, \text{d}x)$; Ye {\cite[Theorems 5.3-5.5]{Ye}} works out this process for compact Riemannian manifolds $(M, g)$, showing in particular that the Sobolev constant obtained does not depend on the geometry of $g$; and Chen {\cite[Appendix B]{ChenThesis}} extends this to asymptotically flat manifolds $(M, g)$ by applying Ye's result to obtain a uniform Sobolev inequality for each member of an exhaustion by compact sets $\Omega_1\subset \Omega_2\subset \dots \subset M$, then extending this to a Sobolev inequality on $M$ by this uniformity in $j$, and the fact that $W^{1,2}_c(M, g)=\bigcup_{j\in \mathbb{N}}W^{1,2}_c(\Omega_j, g)$ is dense in $W^{1, 2}(M, g)$ on complete Riemannian manifolds. We observe that in this argument, the condition of asymptotic flatness can actually be weakened to the condition of $g$ being complete with bounded curvature as in our setting. In summary, we have the following.


\begin{prop}[cf. \cite{Davies}, \cite{Ye}, \cite{ChenThesis}]\label{5.3}
For $n\geq 3$ and $C>0$, there exists a constant $\tilde C(n,C)>0$ such that the following holds. Let $(M^n,h)$, $n\geq 3$ be any complete bounded curvature manifold such that the log-Sobolev inequality
\[
\int_M u^2 \log u^2 \,dV\leq \sigma \left(\int_M |\nabla u|^2+\Psi u^2\,dV\right)-\frac{n}{2}\log \sigma +C
\]
holds for all $\sigma>0$ and all $L^2$-normalized $\|u\|_2=1$ functions $u\in W^{1,2}(M)(M, h)$ and where $\Psi$ is a nonnegative bounded function on $M$. Then the Sobolev inequality
\[
\left(\int_M u^{\frac{2n}{n-2}}\,dV_h\right)^{\frac{n-2}{n}}\leq \tilde C(n,C)\left(\int |\nabla u|^2+\Psi u^2\,dV_h\right)
\]
holds for all $u\in W^{1,2}(M,h)$.
\end{prop}


Now steps 1-3 above can be combined to give a uniform global Sobolev inequality along the flow.


\begin{thm}\label{C}
For $n\geq 3$, there exists a dimensional constant $L(n)>0$ such that the following holds. Suppose that $(M^n,g(t))$, $t\in [0,T]$ is a complete bounded curvature Ricci flow and that the initial metric $g(0)=g_0$ satisfies the conditions of Lemma \ref{logsobinitial}  and the curvature condition
\[
\left(\int_M |R|_{g_0}^{\frac{n}{2}}\,dV_{g_0}\right)^{\frac{2}{n}}\leq \frac{2}{C_{g_0}}.
\]
Then with $A=LC_{g_0}$, the weighted Sobolev inequality  
\begin{equation}\label{Ce1}
\left(\int |u|^{\frac{2n}{n-2}}\,dV_{g(t)}\right)^{\frac{n-2}{n}}\leq A\left(\int |\nabla_{g(t)}u|^2+R_{g(t)}^+u^2\,dV_{g(t)}\right)
\end{equation}
holds for all $u\in W^{1,2}(M^n,g(t))$ on $[0,T]$. 

\end{thm}


\begin{proof}
Part \textbf{(2)} of Lemma \ref{logsobinitial} gives us the following log-Sobolev inequality for the metric $g_0$: 
\begin{align*}
	\int u^2\log u^2\,dV_{g_0} &\leq 4\tau \left(\int |\nabla u|^2+\frac{R}{4}u^2\,dV_{g_0}\right) -\frac{n}{2}\log \tau \\
				   &\quad +\frac{n}{2}\left(\log C_{g_0}+\log \frac{n}{4}-1\right).
\end{align*}
Monotonicity of the $\mu$-functional (Lemma \ref{monotonemufunct}) implies that the uniform log-Sobolev inequality \eqref{sobineqalongflowmufunct} holds for any $t\in [0,T]$ and any $\tau>0$. Now given a $t\in [0,T]$, we replace $R_{g(t)}$ with $R_{g(t)}^+$ in \eqref{sobineqalongflowmufunct} to obtain 
\begin{align}\label{Ce11}
	\int u^2\log u^2\,dV_{g(t)} &\leq 4\tau \left(\int |\nabla u|_{g(t)}^2+\frac{R^+_{g(t)}}{4}u^2\,dV_{g(t)}\right) -\frac{n}{2}\log \tau\\
				\notag     &\quad +\frac{n}{2}\left(\log C_{g_0}+\log \frac{n}{4}-1\right).
\end{align}
We are then in a position to apply Proposition \ref{5.3} which allows us to conclude that \eqref{Ce1} is true at any given time $t$. Since the RHS of \eqref{Ce11} is independent of $t\in [0,T]$, this establishes the result. The fact that $A$ depends linearly on $C_{g_0}$ is realized by tracking the constants in \cite{Ye} - this computation is carried out in the appendix wherein we also provide a bound on $L(n)$. 
\end{proof}


\section{Proof of Theorem \ref{main}} \label{proofofmain}


The proof of Theorem \ref{main} will be broken down into 10 intermediate steps within this section. The first eight of which (Lemma \ref{l4.1} to Proposition~\ref{4.8}) follow the same outline as done by Chen \cite{ChenEric} and the precise naming and numbering of those results are specifically chosen to align with the analogous statements in {\cite[Section 4]{ChenEric}}. The proofs are (naturally) very similar the originals and, in some cases, are identical. For the readers' convenience, we will present the proofs in full when there are notable differences, and will omit them otherwise. \\

We now start the proof of Theorem \ref{main}, and let $(M^n,g)$ satisfy the hypotheses of the theorem for some constant $C_g \in (0,\infty]$. We also let $g(t)$ on $M\times[0, T_{max})$ be the unique complete bounded curvature Ricci flow with $g(0)=g$ with existence guaranteed from \cite{Shi} and uniqueness from \cite{CZ} (here $T_{max}$ is the maximal existence time for the flow and may equal $\infty$).  Note that if $C_g=\infty$, then the integral curvature pinching condition \eqref{mainthme1} implies $|\text{Rm}|_g \equiv 0$ on $M$ and subsequently the conclusion is trivially satisfied by the stationary solution $g(t) \equiv g$.  So without loss of generality, we may assume $C_g<\infty$ in which case we will prove the theorem assuming that the curvature pinching inequality \eqref{mainthme1} is strict, albeit with the slightly larger constant
\[
\delta_1(n):=\frac{1}{39n L(n)}\frac{1}{C_{g}}
\]
 where $L(n)$ is the dimensional constant from Theorem \ref{C}. \\

For ease of reference, we summarize our assumptions as follows:
\begin{equation}\label{section4assumptions}
\begin{cases}
(M^n,g(t)), \; n\geq 3, \; t\in [0,T_{max}), \;g(0)=g, \\
\text{is a complete Ricci flow satisfying:}\\
\text{For any } \; T<T_{max}, \; \sup_{M\times [0,T]}|\text{Rm}|(x,t)<\infty \;\text{ and } \\
\left(\int |\text{Rm}|_{g}^{\frac{n}{2}}\,dV_{g}\right)^{\frac{2}{n}}<\delta_1(n)\frac{1}{C_{g}} \;\text{ where } \;C_g<\infty \;\text{ is such that }\\
\left(\int u^{\frac{2n}{n-2}}\,dV_{g}\right)^{\frac{n-2}{n}}\leq C_g \int |\nabla u|^2\,dV_g \text{ for all } u\in W^{1,2}(M, g).
\end{cases}
\end{equation}


We are now ready to continue with the proof. \\

Our first step is to determine the evolution of  the $L^p$ norm of the curvature. We choose to break up into two different cases based on the value of $p$ because the proof is simpler when $p/2\geq 1$ and may be of independent interest. 


\begin{lem}\label{l4.1} 
Let $(M^n,g(t))$, $t\in [0,T]$ be a complete bounded curvature Ricci flow and fix some smooth function $\phi: M\to \mathbb{R}$ with compact support. Then the following hold. 
\begin{enumerate}[(a)]
\item \label{4.1a} If $\alpha\geq 1$, then there holds
\begin{align*}
	\frac{d}{dt}\int \phi^2 |\text{Rm}|^{2\alpha}\,dV_t &\leq- 2 \int  |\nabla (\phi|\text{Rm}|^\alpha)|^2\,dV_t+(16\alpha+n) \int \phi^2|\text{Rm}|^{2\alpha+1}\,dV_t\\
							    &\quad +2\int |\nabla \phi|^2 |\text{Rm}|^{2\alpha}\,dV_t.
\end{align*}
\item \label{4.1b} If $\frac{1}{2}<\alpha<1$, then for any $\beta>0$ there holds
\begin{align*}
&\frac{d}{dt}\int \phi^2 (|\text{Rm}|^2+\beta)^{\alpha}\,dV_t \leq-C_1(\alpha) \int  |\nabla (\phi(|\text{Rm}|^2+\beta)^{\frac{\alpha}{2}})|^2\,dV_t\\&\;\;\;+(16\alpha+n) \int \phi^2(|\text{Rm}|^2+\beta)^{\alpha+\frac{1}{2}}\,dV_t+C_2(\alpha)\int |\nabla \phi|^2 (|\text{Rm}|^2+\beta)^{\alpha}\,dV_t.
\end{align*}
Moreover, $C_1(\alpha)\to 0$ as $\alpha\to \frac{1}{2}$ and if $\frac{2}{3}<\alpha<1$, then one can take $C_1(\alpha)=\frac{1}{2}$. 
\end{enumerate}
\end{lem}


\begin{proof}
First recall the well-known evolution of $|\text{Rm}|^2$ under the Ricci flow:
\begin{equation}\label{l4.1e1}
\partial_t |\text{Rm}|^2\leq \Delta |\text{Rm}|^2-2|\nabla \text{Rm}|^2+16 |\text{Rm}|^3.
\end{equation}

For part \ref{4.1a}, we combine this with $\partial_t dV_t=-R_tdV_t$, the Kato inequality $|\nabla |T||\leq |\nabla T|$, $|R|\leq n |\text{Rm}|$ (from {\cite[Proposition 7.28]{Lee}}), and $\alpha\geq 1$ which gives
\begin{align*}
&\frac{d}{dt}\int \phi^2|\text{Rm}|^{2\alpha}\,dV_t=\int \phi^2\frac{\partial}{\partial t}\left(|\text{Rm}|^{2\alpha}\,dV_t\right)\\
&=\int \phi^2\left(\alpha|\text{Rm}|^{2(\alpha-1)} \partial_t |\text{Rm}|^2 - R |\text{Rm}|^{2\alpha}\right)\,dV_t\\&\leq
\int \phi^2\alpha|\text{Rm}|^{2(\alpha-1)} \left(\Delta|\text{Rm}|^2 -2|\nabla \text{Rm}|^2+16 |\text{Rm}|^3\right)-\phi^2 R |\text{Rm}|^{2\alpha}\,dV_t\\&\leq
\int \phi^2\Bigl[\alpha|\text{Rm}|^{2(\alpha-1)} \Delta|\text{Rm}|^2 -2\alpha|\text{Rm}|^{2\alpha-2}|\nabla |\text{Rm}||^2\\
																			  &\hspace{3.3em}+\left(16\alpha + n\right) |\text{Rm}|^{2\alpha+1}\Bigr]\,dV_t\\&=
\int \phi^2\alpha|\text{Rm}|^{2(\alpha-1)} \Delta|\text{Rm}|^2 \,dV_t -\frac{2}{\alpha}\int \phi^2 |\nabla |\text{Rm}|^\alpha|^2\,dV_t\\
																			  &\quad +\left(16\alpha + n\right)\int \phi^2 |\text{Rm}|^{2\alpha+1}\,dV_t.
\end{align*}
Now we integrate by parts the Laplacian term using that the support of $\phi$ is compact (so that no boundary integrals are introduced) to obtain
\begin{align*}
&\int \phi^2\alpha|\text{Rm}|^{2(\alpha-1)} \Delta|\text{Rm}|^2 \,dV_t\\
&\quad =- \int \alpha \nabla\left(\phi^2|\text{Rm}|^{2(\alpha-1)}\right) \nabla|\text{Rm}|^2 \,dV_t\\
&\quad =-\int \alpha\nabla \phi^2 |\text{Rm}|^{2(\alpha-1)} \nabla|\text{Rm}|^2 \,dV_t -\int \alpha\phi^2\nabla|\text{Rm}|^{2(\alpha-1)} \nabla|\text{Rm}|^2 \,dV_t\\
&\quad = -2\int  \nabla \phi^2 |\text{Rm}|^\alpha \nabla|\text{Rm}|^\alpha \,dV_t-\int \alpha\phi^2\nabla|\text{Rm}|^{2(\alpha-1)} \nabla|\text{Rm}|^2 \,dV_t\\
&\quad = -4\int  \phi\nabla \phi |\text{Rm}|^\alpha \nabla|\text{Rm}|^\alpha \,dV_t-\int \alpha\phi^2\nabla|\text{Rm}|^{2(\alpha-1)} \nabla|\text{Rm}|^2 \,dV_t\\
&\quad = \textbf{I}-\alpha\, \textbf{II}.
\end{align*}
The first term above is simplified by expanding out the derivative
\begin{align}\label{4.1i}
	\textbf{I}&=-2\int |\nabla (\phi |\text{Rm}|^\alpha)|^2\,dV_t+2\int |\nabla \phi|^2 |\text{Rm}|^{2\alpha}\,dV_t\\
	\notag 	  &\quad +2\int \phi^2 |\nabla |\text{Rm}|^\alpha|^2\,dV_t.
\end{align}
For the second term, we have to be careful because $\nabla|\text{Rm}|^{2(\alpha-1)}$ is singular at points where $|\text{Rm}|=0$ when $\alpha <3/2$ (though this singularity is countered by a contribution of $|\text{Rm}|$ from the $\nabla |\text{Rm}|^2$ term). Since $\alpha\geq 1$, for any $\epsilon>0$ there holds
\begin{align*}
&\nabla(|\text{Rm}|+\epsilon)^{2(\alpha-1)} \nabla(|\text{Rm}|+\epsilon)^2\\
&\quad =4(\alpha-1)(|\text{Rm}|+\epsilon)^{2\alpha-3} (|\text{Rm}|+\epsilon) |\nabla (|\text{Rm}|+\epsilon)|^2\\
&\quad =4(\alpha-1)(|\text{Rm}|+\epsilon)^{2\alpha-2} |\nabla (|\text{Rm}|+\epsilon)|^2\\
&\quad =\frac{4(\alpha-1)}{\alpha^2}|\nabla (|\text{Rm}|+\epsilon)^\alpha|^2.
\end{align*}
Of course the RHS converges to $\frac{4(\alpha-1)}{\alpha^2}|\nabla |\text{Rm}|^\alpha|^2$ pointwise as $\epsilon\to 0$, and since $\phi$ is compactly supported we can apply the Dominated Convergence Theorem (twice) to see that 
\begin{align}
\label{4.1ii}
\textbf{II}&=\lim_{\epsilon \to 0}\int \phi^2\nabla(|\text{Rm}|+\epsilon)^{2(\alpha-1)} \nabla(|\text{Rm}|+\epsilon)^2 \,dV_t\\
\notag &=\lim_{\epsilon\to 0} \frac{4(\alpha-1)}{\alpha^2}\int \phi^2|\nabla (|\text{Rm}|+\epsilon)^\alpha|^2\,dV_t\\
\notag &=\frac{4(\alpha-1)}{\alpha^2}\int \phi^2|\nabla |\text{Rm}|^\alpha|^2\,dV_t.
\end{align}
So using \eqref{4.1i} and \eqref{4.1ii}, we have
\begin{align*}
	\int \phi^2\alpha|\text{Rm}|^{2(\alpha-1)} \Delta|\text{Rm}|^2 \,dV_t &= -4\int  \phi\nabla \phi |\text{Rm}|^\alpha \nabla|\text{Rm}|^\alpha \,dV_t\\
									      &\quad -\frac{4(\alpha-1)}{\alpha}\int \phi^2|\nabla |\text{Rm}|^\alpha|^2\,dV_t.
\end{align*}

Finally, we gather all the $\phi^2 |\nabla |\text{Rm}|^\alpha|^2$ terms to see that
\begin{multline}\label{l4.1final}
\left[-\frac{2}{\alpha}-\frac{4(\alpha-1)}{\alpha}+2 \right]\int \phi^2 |\nabla |\text{Rm}|^\alpha|^2\,dV_t \\
= -\frac{2(\alpha-1)}{\alpha}\int \phi^2 |\nabla |\text{Rm}|^\alpha|^2\,dV_t\leq 0,
\end{multline}
where we have again used that $\alpha\geq 1$. This proves part \ref{4.1a}.\\

\textbf{(b)} The proof is similar to {\cite[Lemma 2.1]{CCL}}, though in that proof it is assumed that $\alpha\geq \frac{n}{4}$ and that $\phi$ is a function of space \emph{and time}. Since our condition on $\alpha$ is different and our $\phi$ is a function of space only, we provide the details for the readers' convenience and to keep track of the constants. We begin just as before:
\begin{align*}
&\frac{d}{dt}\int \phi^2(|\text{Rm}|^2+\beta)^{\alpha}\,dV_t\\
&\quad =\int \phi^2\left(\alpha(|\text{Rm}|^2+\beta)^{\alpha-1} \partial_t |\text{Rm}|^2 - R (|\text{Rm}|^2+\beta)^{\alpha}\right)\,dV_t\\
&\quad \leq \int \phi^2\alpha(|\text{Rm}|^2+\beta)^{\alpha-1} \left(\Delta|\text{Rm}|^2 -2|\nabla |\text{Rm}||^2\right)\\
&\qquad\qquad +(n+16\alpha)\phi^2 (|\text{Rm}|^2+\beta)^{\alpha+\frac{1}{2}}\,dV_t.
\end{align*}
Now we integrate by parts the Laplacian term which gives
\begin{align*}
&\int \phi^2\alpha(|\text{Rm}|^2+\beta)^{\alpha-1} \Delta|\text{Rm}|^2 \,dV_t\\
&\quad =- \int \alpha \nabla\left(\phi^2(|\text{Rm}|^2+\beta)^{\alpha-1}\right) \nabla|\text{Rm}|^2 \,dV_t\\
&\quad = -4\alpha\int  \phi \langle\nabla \phi,\nabla |\text{Rm}|\rangle (|\text{Rm}|^2+\beta)^{\alpha-1} |\text{Rm}|  \,dV_t-4\alpha(\alpha-1)\\
&\qquad \times\int \phi^2(|\text{Rm}|^2+\beta)^{\alpha-2}|\text{Rm}|^2 |\nabla|\text{Rm}||^2 \,dV_t.
\end{align*}
Note that the coefficient in the latter term is actually positive because $\alpha<1$, so we have
\begin{multline*}
-4\alpha(\alpha-1)\int \phi^2(|\text{Rm}|^2+\beta)^{\alpha-2}|\text{Rm}|^2 |\nabla|\text{Rm}||^2 \,dV_t\\
\leq-4\alpha(\alpha-1)\int \phi^2(|\text{Rm}|^2+\beta)^{\alpha-1} |\nabla|\text{Rm}||^2 \,dV_t.
\end{multline*}
For the former term, we apply Cauchy-Schwarz, H\"older's, and Young's inequalities
\begingroup\allowdisplaybreaks\begin{align*}
& -4\alpha\int  \phi \langle\nabla \phi,\nabla |\text{Rm}|\rangle (|\text{Rm}|^2+\beta)^{\alpha-1} |\text{Rm}|  \,dV_t\\
&\quad \leq 4\alpha\int  \phi |\nabla \phi| \,|\nabla |\text{Rm}||\, (|\text{Rm}|^2+\beta)^{\alpha-1} |\text{Rm}|  \,dV_t\\
&\quad \leq  4\alpha\int  \phi |\nabla \phi| \,|\nabla |\text{Rm}||\, (|\text{Rm}|^2+\beta)^{\alpha-\frac{1}{2}}   \,dV_t\\
&\quad \leq  4\alpha\left(\int  \phi^2 (|\text{Rm}|^2+\beta)^{\alpha-1}|\nabla |\text{Rm}||^2 \,dV_t\right)^{\frac{1}{2}}\left(\int|\nabla \phi|^2(|\text{Rm}|^2+\beta)^{\alpha} \,dV_t\right)^{\frac{1}{2}}\\
&\quad \leq 2\alpha\epsilon\int  \phi^2 (|\text{Rm}|^2+\beta)^{\alpha-1}|\nabla |\text{Rm}||^2 \,dV_t\\
&\qquad +\frac{2\alpha}{\epsilon}\int|\nabla \phi|^2(|\text{Rm}|^2+\beta)^{\alpha} \,dV_t=\textbf{I}+\textbf{II}. 
\end{align*}\endgroup
Now we can gather up the terms in the form of \textbf{I} and restrict $\epsilon>0$ small enough so that the coefficient is negative. Then we can use Cauchy-Schwarz again (but in reverse) to see that
\begin{align*}
&\left[-2\alpha+4\alpha(1-\alpha)+2\alpha\epsilon\right]\int  \phi^2 (|\text{Rm}|^2+\beta)^{\alpha-1}|\nabla |\text{Rm}||^2 \,dV_t\\
&\quad \leq \left[-2\alpha+4\alpha(1-\alpha)+2\alpha\epsilon\right]\int  \phi^2 (|\text{Rm}|^2+\beta)^{\alpha-2}|\text{Rm}|^2|\nabla |\text{Rm}||^2 \,dV_t\\
&\quad = \left[-2\alpha+4\alpha(1-\alpha)+2\alpha\epsilon\right]\frac{1}{\alpha^2}\int  \phi^2 |\nabla (|\text{Rm}|^2+\beta)^{\frac{\alpha}{2}}|^2 \,dV_t\\
&\quad \leq \left[-2\alpha+4\alpha(1-\alpha)+2\alpha\epsilon\right]\\
&\quad \quad \times \frac{1}{\alpha^2}\left(\frac{1}{2}\int  |\nabla \phi (|\text{Rm}|^2+\beta)^{\frac{\alpha}{2}}|^2 \,dV_t-\int  |\nabla \phi|^2 (|\text{Rm}|^2+\beta)^{\alpha}\,dV_t\right).
\end{align*}
Now whenever $\alpha>\frac{1}{2}$, we can restrict $\epsilon$ small enough so that 
\[
-C_1(\alpha):=\left[-2\alpha+4\alpha(1-\alpha)+2\alpha\epsilon\right]\frac{1}{2\alpha^2}<0,
\]
and if $\alpha>\frac{2}{3}$, $\epsilon$ can be taken small enough so that $C_1(\alpha)=\frac{1}{2}$. This proves the result.
\end{proof}


We now prove that our hypotheses of \eqref{section4assumptions} are enough to ensure that the scale-invariant integral curvature quantity $\int |\text{Rm}|^{\frac{n}{2}}\,dV_t$ is nonincreasing in time. The following Lemma \ref{l4.2} is similar to {\cite[Lemma 2.2]{ChanLee}}. In fact the necessary tools needed in both proofs are the same: (a) a weighted Sobolev inequality with scalar curvature term along the flow; and (b) smallness of the scale-invariant integral curvature $\int |\text{Rm}|^{\frac{n}{2}} \,dV_t$ \emph{uniformly along the flow} relative to the constant in the Sobolev inequality. In \cite{ChanLee} when $n\geq 4$, both (a) and (b) are guaranteed by a rescaling argument from the short-time existence result in \cite{CHL}. For us however, we obtain (a) from Theorem \ref{C} while (b) is guaranteed for some short time interval by the upper-continuity of $t\mapsto \int |\text{Rm}|^{\frac{n}{2}}\,dV_t$ from Proposition \ref{finitenessofn/2}, which will be sufficient for our proof below. Since our hypotheses and some lines of argument differ from those in \cite{ChanLee}, we will include the proof for completeness. We will also refer to these steps in establishing the case for $n=3$, and as well as Lemma \ref{l4.4} which extends this monotonicity result to other powers of $|\text{Rm}|$.


\begin{lem}\label{l4.2} 
Let $(M^n,g(t))$ satisfy the conditions in \eqref{section4assumptions}. Then $t\mapsto \int |\text{Rm}|^{\frac{n}{2}}\,dV_t$ is nonincreasing on $[0,T_{max})$.
\end{lem}


\begin{proof}
We first prove the result for dimensions $n\geq 4$ and comment on the necessary modification for dimension $n=3$ at the end. Fix some $T<T_{max}$. It clearly suffices to show that $t\mapsto \int |\text{Rm}|^{\frac{n}{2}}\,dV_t$ is nonincreasing on $[0,T]$. \\

Let $\phi: M\to \mathbb{R}$ be some smooth compactly supported function to be chosen later. Lemma \ref{l4.1} part \ref{4.1a} (applied with $\alpha=\frac{n}{4}$) gives
\begin{align*}
	\frac{d}{dt}\int \phi^2|\text{Rm}|^{\frac{n}{2}}\,dV_t &\leq -2 \int |\nabla (\phi|\text{Rm}|^{\frac{n}{4}})|^2\,dV_t+(4n+n) \int \phi^2|\text{Rm}|^{\frac{n}{2}+1}\,dV_t\\
							       &\quad +2\int |\nabla \phi|^2 |\text{Rm}|^{\frac{n}{2}}\,dV_t.
\end{align*}
Note that $\delta_1(n)<\frac{2}{n}$ so that the conclusion of Theorem \ref{C} holds with $A=L(n)C_g$, thus giving
\begin{align*}
&-2\int |\nabla (\phi|\text{Rm}|^{\frac{n}{4}})|^2\,dV_t\\
&\quad \leq -\frac{2}{A} \left(\int (\phi|\text{Rm}|^{\frac{n}{4}})^{\frac{2n}{n-2}}\,dV_t\right)^{\frac{n-2}{n}}+2\int R^+ \phi^2|\text{Rm}|^{\frac{n}{2}}\,dV_t\\
&\quad \leq  -\frac{2}{A} \left(\int (\phi|\text{Rm}|^{\frac{n}{4}})^{\frac{2n}{n-2}}\,dV_t\right)^{\frac{n-2}{n}}+2n\int \phi^2|\text{Rm}|^{\frac{n}{2}+1}\,dV_t. 
\end{align*}
Now we gather the terms with $|\text{Rm}|^{\frac{n}{2}+1}$ in the integrand and apply H\"older's inequality to obtain
\[
7n\int \phi^2|\text{Rm}|^{\frac{n}{2}+1}\,dV_t\leq 7n\left(\int |\text{Rm}|^{\frac{n}{2}}\,dV_t\right)^{\frac{2}{n}}\left(\int \phi^{\frac{2n}{n-2}} |\text{Rm}|^{\frac{n}{4}\frac{2n}{n-2}}\,dV_t\right)^{\frac{n-2}{n}}.
\]
To summarize the above, for any compactly supported smooth $\phi$ we have 
\begin{align}\label{4.2e2}
&\frac{d}{dt}\int \phi^2|\text{Rm}|^{\frac{n}{2}}\,dV_t \\
\notag &\quad \leq \left(\int \phi^{\frac{2n}{n-2}}|\text{Rm}|^{\frac{n}{4}\frac{2n}{n-2}}\,dV_t\right)^{\frac{n-2}{n}}\left[-\frac{2}{A} +7n\left(\int |\text{Rm}|^{\frac{n}{2}}\,dV_t\right)^{\frac{2}{n}}\right]\\
\notag &\qquad +2\int |\nabla \phi|^2 |\text{Rm}|^{\frac{n}{2}}\,dV_t.
\end{align}

The choice of $\phi$ we will be employing is as follows. Let $\psi: \mathbb{R} \to [0,1]$ be a smooth decreasing cutoff function which satisfies $\psi\equiv 1$ on $(-\infty, 1]$ and $\psi\equiv 0$ on $[2, \infty)$. Fix some $x_0\in M$ and define, for $r>0$, 
\[
\phi_r(x)= \psi\left(\frac{\tilde d_0(x)}{r}\right).
\]
Here $\tilde d_0 $ is smooth with uniformly bounded gradient $\sup_{x\in M}|\nabla \tilde d_0|_g\leq c$, and uniformly bounded from below by the distance function from $x_0\in M$ relative to $g_0$ in the sense that $c^{-1} d_{g_0}(x_0, x) \leq \tilde d_0(x)$ for all $x\in M$ for some $c>0$ depending on $g_0$ (in particular on a lower bound of $\text{Rc}_{g_0}$, see {\cite[Theorem 4.2]{SY}}).  From this, assumption \eqref{section4assumptions}, and Proposition \ref{finitenessofn/2}, we can define the following constants which may depend on $T<T_{max}$ :
\begin{align*}
	C&:=\sup_{r\geq 1} \sup_{M\times [0,T]} r^2 |\nabla^{g(t)}\phi_r|_{g(t)}^2<\infty;\,\,\,\,K:=\sup_{M\times [0,T]}|\text{Rm}|<\infty;\\
	\beta&:=\sup_{t\in [0,T]}\int
|\text{Rm}|^{\frac{n}{2}}\,dV_t<\infty.
\end{align*}

Now we define the function 
\[
F(r, t):=\int \phi_r^2|\text{Rm}|^{\frac{n}{2}}\,dV_t
\]
where  $r\geq 1$ and $t\in [0, T]$ and define 
\[
T':=\sup \{t\in [0, T] : F(\infty, t) \leq (\delta_1/C_g)^{\frac{n}{2}}\}>0
\]
 where the non-emptyness of the set on the RHS follows from assumption~\eqref{section4assumptions} while the positivity of $T'$ follows from Proposition \ref{finitenessofn/2} (in particular, the upper-continuity of $t\mapsto \int |\text{Rm}|^{\frac{n}{2}}\,dV_t$).\\

Now using the fact that $[-\frac{2}{A} + 7n \frac{\delta_1}{C_{g(0)}}]<0$ which follows from $A=L(n)C_{g(0)}$ and our choice of $\delta_1$ in assumption \eqref{section4assumptions}, and our definition of $C, \beta$ above, the evolution \eqref{4.2e2} then gives
\[
\frac{d}{dt}F(r, t) \leq 2\frac{C}{\beta r^2}
\]
for all $r\geq 1$ and all $t\in [0, T']$.  In particular, for all $0<s_1<s_2<T'$ this gives
\begin{align*}
	\int |\text{Rm}|^{\frac{n}{2}}\,dV_{s_2}&=\lim_{r\to \infty}F(r, s_2) \leq \lim_{r\to \infty}\left( F(r, s_1)+ 2\frac{C (s_2 - s_1)}{\beta r^2}\right)\\
						&=F(\infty, s_1)=\int |\text{Rm}|^{\frac{n}{2}}dV_{s_1}
\end{align*}
so that $\int |\text{Rm}|^{\frac{n}{2}}\,dV_{t}$ is non-increasing for $t\in [0, T']$. Here we have used that $\tilde d_0(x) \geq c^{-1} d_g(x,x_0)$ and thus $\{\phi_r =1\}$ exhausts $M$ as $r\to \infty$. But since $\int |\text{Rm}|^{\frac{n}{2}}\,dV_{0}< (\delta_1/C_{g(0)})^{\frac{n}{2}}$, this just imples that $T'=T$ so that $\int |\text{Rm}|^{\frac{n}{2}}\,dV_{t}$ is nonincreasing for $t\in [0, T]$ which concludes the proof of the lemma for $n\geq 4$. \\

For the dimension $n=3$ case, we proceed just as we did above with part \ref{4.1b} of Lemma \ref{l4.1} instead of part \ref{4.1a}. The analogue of the evolution inequality \eqref{4.2e2} that we obtain in this case is
\begin{align*}
&\frac{d}{dt}\int \phi^2 (|\text{Rm}|^2+\beta)^{\alpha}\,dV_t \\
&\quad \leq C_2(3/4)\int |\nabla \phi|^2 (|\text{Rm}|^2+\beta)^{\alpha}\,dV_t\\
&\qquad +\left(\int \phi^6 (|\text{Rm}|^2+\beta)^{\frac{9}{4}}\,dV_t\right)^{\frac{1}{3}}\\
&\qquad\quad \times\left[-\frac{1}{2A}+\left(15+\frac{3}{2}\right)\left(\int_{\text{supp} \phi} (|\text{Rm}|^2+\beta)^{\frac{3}{4}}\,dV_t\right)^{\frac{2}{3}}\right].
\end{align*}
Now for fixed $r\gg 0$ and $\beta>0$, we can integrate the above inequality on the interval $[t_0,t_1]$, then send $\beta\to 0$ (which is validated by the precense of the cutoff function $\phi$) to obtain 
\begin{align}
\label{jfjfjf}
&\int \phi^2 |\text{Rm}|^{\frac{3}{2}}\,dV_{t_1}-\int \phi^2 |\text{Rm}|^{\frac{3}{2}}\,dV_{t_0}\\
\notag &\quad \leq C_2(3/4)\int_{t_0}^{t_1}\int |\nabla \phi|^2 |\text{Rm}|^{\frac{3}{2}}\,dV_t\,dt\\
\notag &\qquad +\int_{t_0}^{t_1}\left(\int \phi^6 |\text{Rm}|^{\frac{9}{2}}\,dV_t\right)^{\frac{1}{3}}\\
\notag &\qquad\quad\times \left[-\frac{1}{2A}+\left(15+\frac{3}{2}\right)\left(\int_{\text{supp} \phi} |\text{Rm}|^{\frac{3}{2}}\,dV_t\right)^{\frac{2}{3}}\right]\,dt.
\end{align}

From here, one can show $t\mapsto \int|\text{Rm}|^{\frac{3}{2}}\,dV_t$ is nonincreasing just as in the higher dimensional case above. In the process, we use the fact that 
\[
\delta_1(3)<\frac{1}{33 L(3)},
\]
so that the bracked term in \eqref{jfjfjf} is nonpositive.
\end{proof}


The following corollary is not essential to our proof (any time it is used we could use Theorem \ref{C} and Lemma \ref{l4.2} instead) but we include its statement for ease of reference. The proof is identical to the proof of {\cite[Corollary 4.3]{ChenEric}}.


\begin{cor}\label{l4.3} 
Let $(M^n,g(t))$ satisfy the conditions in \eqref{section4assumptions}. Then the uniform Sobolev inequality 
\begin{equation}\label{4.3e2}
\left(\int u^{\frac{2n}{n-2}}\,dV_t\right)^{\frac{n-2}{n}} \leq A_0\int |\nabla u|^2\,dV_t
\end{equation}
holds for all $t\in [0,T_{max})$ and $u\in W^{1,2}(M,g(t))$. Moreover we can take $A_0=2L(n) C_{g_0}$. 
\end{cor}


\begin{lem}\label{l4.4} 
Let $(M^n,g(t))$ satisfy the conditions in \eqref{section4assumptions}. Then $t\mapsto \int |\text{Rm}|^q\,dV_t$ is nonincreasing on $[0,T_{max})$ for any $q$ satisfying $\max\{2, \frac{n}{2}\}\leq q\leq \frac{n}{2}\frac{n}{n-2}$. 
\end{lem}


\begin{proof}
We proceed in a similar fashion as Lemma \ref{l4.2}. Fix $T<T_{max}$ and note that Lemma \ref{l4.2} combined with our bounded curvature assumption implies
\[
\sup_{t\in [0,T]} \|\text{Rm}\|_{L^{q}_{g(t)}}<\infty.
\]

We apply Lemma \ref{l4.1} part \textbf{(a)} with $\alpha=q/2$ (note that $\alpha\geq 1$) and H\"older's inequality to see
\begin{align*}
&\frac{d}{dt}\int \phi^2|\text{Rm}|^{2\alpha}\,dV_t\\
&\quad \leq -2 \int |\nabla (\phi|\text{Rm}|^{\alpha})|^2\,dV_t\\
&\qquad +2\int |\nabla \phi|^2 |\text{Rm}|^{2\alpha}\,dV_t+(16\alpha+n) \int \phi^2|\text{Rm}|^{2\alpha+1}\,dV_t\\
&\quad \leq  -2 \int |\nabla (\phi|\text{Rm}|^{\alpha})|^2\,dV_t+2\int |\nabla \phi|^2 |\text{Rm}|^{2\alpha}\,dV_t\\
&\qquad  +(16\alpha+n) \left(\int |\text{Rm}|^{\frac{n}{2}}\,dV_t\right)^{\frac{2}{n}}\left(\int\phi^{\frac{2n}{n-2}} |\text{Rm}|^{\alpha\frac{2n}{n-2}}\,dV_t\right)^{\frac{n-2}{n}}.
\end{align*}
Now Corollary \ref{l4.3} gives
\[
-2 \int |\nabla (\phi|\text{Rm}|^{\alpha})|^2\,dV_t\leq -\frac{2}{A_0}\left(\int \phi^{\frac{2n}{n-2}}|\text{Rm}|^{\alpha\frac{2n}{n-2}}\,dV_t\right)^{\frac{n-2}{n}} .
\] 
We will also retain from Lemma \ref{l4.2} the notation of $\phi_r$ and the constant $C$ for which there holds 
\[
|\nabla^{g(t)} \phi_r|_{g(t)}^2 \leq \frac{C}{r^2}.
\]
Of course, the constant $C$ here will depend on the bound of the curvature on $M\times [0,T]$. Then plugging in $A_0=2LC_{g_0}$, applying Lemma \ref{l4.2}, and using our definition of $\delta_1(n)$ yields
\begin{align*}
&\frac{d}{dt}\int \phi_r^2|\text{Rm}|^{q}\,dV_t \\
& \leq \left(\int \phi_r^{\frac{2n}{n-2}}|\text{Rm}|^{q\frac{n}{n-2}}\,dV_t\right)^{\frac{n-2}{n}}\left[-\frac{2}{2LC_{g_0}}+(8q +n)\left(\int |\text{Rm}|^{\frac{n}{2}}\,dV_t\right)^{\frac{2}{n}}\right]\\
& \quad +2\int |\nabla \phi_r|^2 |\text{Rm}|^{2\alpha}\,dV_t\\
&  \leq  \left(\int |\text{Rm}|^{q\frac{n}{n-2}}\,dV_t\right)^{\frac{n-2}{n}}\left[-\frac{1}{LC_{g_0}}+\left(\frac{4n^2}{n-2} +n\right)\frac{\delta_1(n)}{C_{g_0}}\right]+2\frac{C}{r^2}\int  |\text{Rm}|^{2\alpha}\,dV_t\\
& \leq  2\frac{C}{r^2}\sup_{t\in [0,T]}\int  |\text{Rm}|^{q}\,dV_t,
\end{align*}
where we have used the fact that $\frac{4n^2}{n-2}+n<39n$. Thus for fixed $0\leq t_0<t_1\leq T$, we may integrate the above from $t_0$ to $t_1$ to see 
\begin{align*}
\int \phi_r^2 |\text{Rm}|^{q}\,dV_{t_1}&\leq \int \phi_r^2 |\text{Rm}|^{q}\,dV_{t_0}+2(t_1-t_0)\frac{C}{r^2}\sup_{t\in [0,T]}\int  |\text{Rm}|^{q}\,dV_t\\&\leq \int |\text{Rm}|^{q}\,dV_{t_0}+2(t_1-t_0)\frac{C}{r^2}\sup_{t\in [0,T]}\int  |\text{Rm}|^{q}\,dV_t.
\end{align*}
Taking $r\to \infty$ completes the proof. 
\end{proof}


\begin{prop}\label{4.5} 
Let $(M^n,g(t))$ satisfy the conditions in \eqref{section4assumptions}. Then the curvature estimate 
\[
\sup_{x\in M}|\text{Rm}|^2(x,t)\leq \max( c_1(g_0) t^{-\frac{2(n-2)}{n}}, c_2(g_0) t^{\frac{4(n-2)}{n^2}})
\]
holds for all $t\in [0,T_{max})$. Here $c_1(g_0)$ and $c_2(g_0)$ are positive constants depending only on the initial metric $g_0$. In particular, they depend only on $C_{g_0}$ and the initial curvature quantity $\int |\text{Rm}|^{\frac{n}{2}\frac{n}{n-2}}\,dV_0$. 
\end{prop}


\begin{proof}
Like the previous lemmas, we will fix some $T<T_{max}$ and prove the estimate on $[0,T]$. For convenience of notation, let $f(x,t)=|\text{Rm}|^2(x,t)$ and write $\phi$ to be some $\phi_r$ just as in the notation used in Lemma \ref{l4.2} (we'll specify $r$ later). By Lemma \ref{l4.1}, for any $p\geq \frac{n}{4}\frac{n}{n-2}$ (note that this condition forces $p\geq 1$), there holds
\begin{align}\label{4.5e1}
&\frac{d}{dt}\int \phi^2f^p \,dV_t+\int |\nabla (\phi f^{\frac{p}{2}})|^2\,dV_t\\
\notag &\quad \leq -\int |\nabla (\phi f^{\frac{p}{2}})|^2\,dV_t+(16p+n) \int \phi^2|\text{Rm}|f^p\,dV_t\\
\notag &\qquad +2\int |\nabla \phi|^2 f^p\,dV_t.\\
\notag &\quad =:\textbf{I}+(16p+n)\textbf{II}+\textbf{III}.
\end{align}
We will estimate the above terms $\textbf{I},\textbf{II},\textbf{III}$ separately.  \\

Corollary \ref{l4.3} immediately yields
\begin{equation}\label{4.5e2}
\textbf{I} \leq -\frac{1}{A_0} \left(\int \phi^{\frac{2n}{n-2}}f^{p\frac{n}{n-2}}\,dV_t\right)^{\frac{n-2}{n}}
\end{equation}
where $A=2L(n) C_{g_0}$. To estimate $\textbf{II}$, we recall the generalized H\"older's inequality: 
\[
\left\|\prod_{j=1}^k f_j\right\|_{L^1}\leq \prod_{j=1}^k \|f_j\|_{L^{r_j}} \;\text{ where } \; \sum_{j=1}^k \frac{1}{r_i}=1.
\]
We will apply this with 
\[
\begin{cases}
f_1=|\text{Rm}|,\;\; f_2=\eta^{-\left(1-\frac{n}{2q}\right)\frac{n}{2q}}\phi^{2\left(1-\frac{n}{2q}\right)}f^{p\left(1-\frac{n}{2q}\right)}, \;\; f_3=\eta^{\left(1-\frac{n}{2q}\right)\frac{n}{2q}}\phi^{\frac{n}{q}}f^{p\frac{n}{2q}},\\
r_1=q, \;\; r_2=\frac{1}{1-\frac{n}{2q}}, \; \; r_3=\frac{2q}{n}\frac{n}{n-2},
\end{cases} 
\]
where $q=\frac{n}{2}\frac{n}{n-2}$ and $\eta>0$ will be chosen later. This gives
\begin{align*}
	\textbf{II} &\leq \left(\int |\text{Rm}|^q \,dV_t\right)^{\frac{1}{q}} \left(\eta^{-\frac{n}{2q}}\int \phi^2 f^p \,dV_t\right)^{1-\frac{n}{2q}}\\
		    &\quad\times \left(\eta^{\left(1-\frac{n}{2q}\right)\frac{n}{n-2}}\int \phi^{\frac{2n}{n-2}}f^{p\frac{n}{n-2}} \,dV_t\right)^{\frac{n-2}{n}\frac{n}{2q}}.
\end{align*}
Now we apply Lemma \ref{l4.4}, which states that $t\mapsto\|\text{Rm}\|_{L^q_{g(t)}}$ is nonincreasing along the flow. So writing $\beta=\|\text{Rm}\|_{L^q_{g(0)}}$, we have that $\|\text{Rm}\|_{L^q_{g(t)}}\leq \beta$ for all $t\geq 0$. Now to the remaining terms we apply Young's inequality 
\[
ab\leq \frac{a^{r_1}}{r_1}+\frac{b^{r_2}}{r_2}, \;\; \text{ where } \frac{1}{r_1}+\frac{1}{r_2}=1,
\]
with $r_1=\frac{1}{1-\frac{n}{2q}}$, $r_2=\frac{2q}{n}$ to see that 
\begin{align}
\textbf{II}&\leq \beta\left(\eta^{-\frac{n}{2q}}\int \phi^2 f^p \,dV_t\right)^{1-\frac{n}{2q}} \left(\eta^{\left(1-\frac{n}{2q}\right)\frac{n}{n-2}}\int \phi^{\frac{2n}{n-2}}f^{\frac{p}{2}\frac{2n}{n-2}} \,dV_t\right)^{\frac{n-2}{n}\frac{n}{2q}}\nonumber\\&\leq 
\beta\left(1-\frac{n}{2q}\right)\eta^{-\frac{n}{2q}}\int \phi^2 f^p \,dV_t+ \beta\frac{n}{2q}\eta^{1-\frac{n}{2q}}\left(\int \phi^{\frac{2n}{n-2}}f^{p\frac{n}{n-2}} \,dV_t\right)^{\frac{n-2}{n}}.\label{4.5e3}
\end{align}

Finally, to estimate $\textbf{III}$, we can use Lemma \ref{l4.2} combined with the fact that $(M, g(t))$, $t\in [0,T]$ is a \emph{bounded curvature} flow to obtain
\begin{equation}\label{4.5e3a}
\sup_{t\in [0,T]}\int f^p\,dV_t <\infty.
\end{equation}
Then following just as in the proof of the previous lemmas, for any $\epsilon>0$, there exists $r_\epsilon\gg 0$ such that whenever $r\geq r_\epsilon$ and $t\in [0,T]$ there holds 
\begin{equation}\label{4.5e4}
\textbf{III}<\epsilon.
\end{equation}

Now putting \eqref{4.5e2}, \eqref{4.5e3} and \eqref{4.5e4} back into \eqref{4.5e1} we have
\begin{align*}
\frac{d}{dt}\int \phi^2f^p \,dV_t&+\int |\nabla (\phi f^{\frac{p}{2}})|^2\,dV_t\leq \beta(16p+n)\left(1-\frac{n}{2q}\right)\eta^{-\frac{n}{2q}}\int \phi^2f^p \,dV_t\\&\;\;\;+\left[\frac{\beta n(16p+n)}{2q}\eta^{1-\frac{n}{2q}}-\frac{1}{A_0}\right]\left(\int \phi^{\frac{2n}{n-2}}f^{p\frac{n}{n-2}} \,dV_t\right)^{\frac{n-2}{n}}+\epsilon.
\end{align*}
Now we choose $\eta>0$ so that the bracketed quantity above is zero. Precisely, we take
\[
\eta^{1-\frac{n}{2q}}=\frac{2q}{\beta n(16p+n)A_0},
\]
so that plugging this value of $\eta$ back into the above, and simplifying gives
\begin{align}
\label{4.5e5}
&\frac{d}{dt}\int \phi^2f^p \,dV_t+\int |\nabla (\phi f^{\frac{p}{2}})|^2\,dV_t\\
\notag &\ \ \leq \beta(16p+n)\left(1-\frac{n}{2q}\right)\left(\frac{2q}{\beta n(16p+n)A_0}\right)^{-\frac{2q}{2q-n}\frac{n}{2q}}\int \phi^2 f^p \,dV_t +\epsilon \\
\notag &\ \ = \beta(16p+n)\frac{2}{n}\left(\frac{A_0(n-2)\beta(16p+n)}{n}\right)^{\frac{n-2}{2}}\int  \phi^2 f^p \,dV_t+\epsilon \\
\notag &\ \  \leq (20p)^{\frac{n}{2}} \beta^{\frac{n}{2}}A_0^{\frac{n-2}{2}}\int  \phi^2 f^p \,dV_t+\epsilon =:
K_{p}\int \phi^2 f^p\,dV_t+\epsilon,
\end{align}
where in the second line we used our value of $q=\frac{n}{2}\frac{n}{n-2}$, and in the last line we used the estimates $6p+n\leq 20p$ and $n\geq 3$. Now for $0<\tau<\tau'\leq T$, define $\psi: [0,T]\to [0,1]$ by 
\[
\psi(t)=\begin{cases}
0 & 0\leq t\leq \tau\\
\frac{t-\tau}{\tau'-\tau} & \tau\leq t\leq \tau'\\
1 & \tau'\leq t\leq T
\end{cases}.
\]
Now multiply \eqref{4.5e5} by $\psi$ and find that 
\[
\frac{d}{dt}\left(\psi\int \phi^2 f^p \,dV_t\right)+\psi\int |\nabla (\phi f^{\frac{p}{2}})|^2\,dV_t\leq (K_p\psi+\psi')\int \phi^2 f^p \,dV_t+\epsilon,
\]
so that after integrating from $\tau$ to $\tilde t\in [\tau',T]$, we obtain
\begin{align*}
&\int \phi^2 f^p \,dV_{\tilde t}+ \int_{\tau'}^{\tilde t}\int |\nabla (\phi f^{\frac{p}{2}})|^2\,dV_t\,dt\\
&\quad \leq \psi(\tilde t)\int \phi^2 f^p \,dV_{\tilde t}-\psi(\tau) \int \phi^2 f^p\,dV_\tau+ \int_\tau^{\tilde t}\psi\int |\nabla (\phi f^{\frac{p}{2}})|^2\,dV_t\,dt\\
&\quad = \int_\tau^{\tilde t}\left[\frac{d}{dt}\left(\psi\int \phi^2 f^p \,dV_t\right)+\psi\int |\nabla (\phi f^{\frac{p}{2}})|^2\,dV_t\right]\,dt\\
&\quad \leq  \int_\tau^{\tilde t}(K_p\psi+\psi')\int \phi^2 f^p \,dV_t\,dt+(\tilde t-\tau)\epsilon\\
&\quad \leq  \left(K_p+\frac{1}{\tau'-\tau}\right)\int_\tau^{T}\int  f^p \,dV_t\,dt+T\epsilon.
\end{align*}
In particular, because both terms on the LHS are nonnegative, the above yields two inequalities which will be used later. The first is 
\begin{multline}\label{4.5e8}
\int_{\tau'}^{T}\int |\nabla (\phi f^{\frac{p}{2}})|^2\,dV_t\,dt\\
\leq 
\left(K_p+\frac{1}{\tau'-\tau}\right)\int_\tau^{T}\int f^p \,dV_t\,dt+T\epsilon \; \text{ for any } \epsilon>0, r\geq r_{\epsilon}. 
\end{multline}
Likewise, considering the other term on the LHS gives
\[
\int \phi^2 f^p \,dV_{\tilde t}\leq\left(K_p+\frac{1}{\tau'-\tau}\right)\int_\tau^{T}\int  f^p \,dV_t\,dt +T\epsilon
\]
for any $\tilde t\in [\tau',T]$. Since the integrand $\phi^2 f^p$ is nonnegative however, we can take $r\to \infty$, followed by taking $\epsilon\to 0$ which gives the second useful inequality:
\begin{equation}\label{4.5e7}
\int f^p \,dV_{\tilde t}\leq\left(K_p+\frac{1}{\tau'-\tau}\right)\int_\tau^{T}\int  f^p \,dV_t\,dt.
\end{equation}

For $p\geq 1$ and $\tau\in [0,T]$, we define
\[
H(p,\tau)=\int_\tau^{T}\int f^p \,dV_t\,dt,
\]
and let $\nu=1+\frac{2}{n}$. We need to prove the following claim:\\

\textbf{Claim:} For $0\leq \tau<\tau'\leq T$, there holds 
\begin{equation}\label{4.5e9}
H(\nu p, \tau')\leq A_0 \left(K_p+\frac{1}{\tau'-\tau}\right)^\nu H(p,\tau)^\nu.
\end{equation}

\begin{proof}[Proof of Claim:]
We will first show the desired inequality albeit with a $\phi^2$ term in the integrand on the LHS. With this in mind, we apply H\"older's inequality followed by Corollary \ref{l4.3} to see that
\begin{align*}
\int_{\tau'}^{T}\int \phi^2 f^{p\left(1+\frac{2}{n}\right)} \,dV_t\,dt&\leq 
\int_{\tau'}^{T}\left(\int_{\text{supp} \phi} f^p \,dV_t\right)^{\frac{2}{n}}\left(\int(\phi f^{\frac{p}{2}})^{\frac{2n}{n-2}}\,dV_t\right)^{\frac{n-2}{n}}\,dt\\&\leq 
A_0\left(\sup_{\tau'\leq t\leq T}\int_{\text{supp} \phi} f^p \,dV_t\right)^{\frac{2}{n}}\int_{\tau'}^{T}\int |\nabla (\phi f^{\frac{p}{2}})|^2\,dV_t\,dt.
\end{align*}
Applying \eqref{4.5e8}, and noting that the supremum above is achieved at some  $\tilde t\in [\tau',T]$ (of course, the time $\tilde t$ possibly depends on $\phi$), we can continue our estimate above to 
\begin{align*}
&\int_{\tau'}^{T}\int \phi^2 f^{p\left(1+\frac{2}{n}\right)} \,dV_t\,dt\\
&\quad \leq A_0\left(\int f^p \,dV_{\tilde
t}\right)^{\frac{2}{n}}\left[\left(K_p+\frac{1}{\tau'-\tau}\right)\int_\tau^{T}\int
f^p \,dV_t\,dt+T\epsilon\right]\\
&\quad \leq
A_0\left(\int f^p \,dV_{\tilde t}\right)^{\frac{2}{n}}\left(K_p+\frac{1}{\tau'-\tau}\right)H(p,\tau)+A_0T\epsilon\left(\int f^p \,dV_{\tilde t}\right)^{\frac{2}{n}}\\
&\quad \leq
A_0\left(K_p+\frac{1}{\tau'-\tau}\right)^{\nu}H(p,\tau)^{\nu}+A_0T\epsilon\left(\sup_{t\in [0,T]}\int f^p \,dV_{ t}\right)^{\frac{2}{n}}.
\end{align*}

Now the RHS is independent of $r\geq r_\epsilon$, so we can take $r\to \infty$, apply the Monotone Convergence Theorem, then take $\epsilon\to 0$. Note that the quantity $\sup_{t\in [0,T]}\int f^p\,dV_t$ is finite by \eqref{4.5e3a} above. This proves the claim.
\end{proof}


At this point, the remainder of the proof is verbatim to {\cite[Proposition 4.5]{ChenEric}}. We will however provide an outline to compute the constants involved and for our future reference. Let $p_0=\frac{q}{2}=\frac{n}{4}\frac{n}{n-2}$, recall $\nu=1+\frac{2}{n}$ and define
\[
\eta=\nu^{\frac{n}{2}}, \;\; p_k=\nu^kp_0, \; \; \tau_k=(1-\eta^{-k})T, \;\; \Phi_k=H(p_k,\tau_k)^{\frac{1}{p_k}}.
\]
So applying \eqref{4.5e9}, we get
\[
\Phi_{k+1}\leq A_0^{\frac{1}{\nu p_k}} \left(\beta^{\frac{n}{2}} A_0^{\frac{n-2}{2}}(20p_0)^{\frac{n}{2}} +\frac{1}{(1-\eta^{-1})T}\right)^{\frac{1}{p_k}}\eta^{\frac{k}{p_k}} \Phi_k
\]
and iterating the above gives
\begin{equation}\label{4.5e10}
\lim_{k\to \infty}\Phi_{k}\leq A_0^{\frac{1}{\nu}\sum_{k=0}^\infty\frac{1}{ p_k}} \left(\beta^{\frac{n}{2}} A_0^{\frac{n-2}{2}}(20p_0)^{\frac{n}{2}}+\frac{1}{(1-\eta^{-1})T}\right)^{\sum_{k=0}^\infty\frac{1}{p_k}}\eta^{\sum_{k=0}^\infty \frac{k}{p_k}} \Phi_0.
\end{equation}
We can also calculate
\[
\sum_{k=0}^\infty \frac{1}{p_k}=\frac{2(n^2-4)}{n^2}, \;\; \frac{1}{\nu}\sum_{k=0}^\infty\frac{1}{ p_k}=\frac{2(n-2)}{n}, \;\;  \sum_{k=0}^\infty \frac{k}{p_k}=\frac{n^2-4}{n}.
\]
Finally, 
\[
\eta^{\frac{n^2-4}{n}}=\nu^{\frac{n}{2}\frac{n^2-4}{n}}=\left(1+\frac{2}{n}\right)^{\frac{n^2-4}{2}}\leq e^{n-1} 
\]
because the function $\left(1+\frac{2}{x}\right)^{\frac{x^2-4}{2}}e^{1-x}$ is increasing on $[0,\infty)$ and bounded by 1. Plugging these identities and $p_0=\frac{n}{4}\frac{n}{n-2}\leq n$ back into \eqref{4.5e10} gives
\begin{equation}\label{4.5e11}
\lim_{k\to \infty}\Phi_{k}\leq e^{n-1}A_0^{\frac{2(n-2)}{n}} \left(\beta^{\frac{n}{2}} A_0^{\frac{n-2}{2}}(20n)^{\frac{n}{2}}+\frac{1}{(1-\eta^{-1})T}\right)^{\frac{2(n^2-4)}{n^2}} \Phi_0
\end{equation}

Just as in \cite{ChenEric}, there holds 
\[
\Phi_0 \leq T^{\frac{4(n-2)}{n^2}}\beta^2 \;\text{ and } \; \lim_{k\to \infty}\Phi_k \geq \|f\|_{L^\infty_{g(T)}}.
\]

Using this in \eqref{4.5e11} gives
\begin{align*}
&\sup_{x\in M}|\text{Rm}|^2_{g(T)} =\|f\|_{L^{\infty}_{g(T)}}\\
&\quad \leq e^{n-1}A_0^{\frac{2(n-2)}{n}} \left(\beta^{\frac{n}{2}} A_0^{\frac{n-2}{2}}(20n)^{\frac{n}{2}}+\frac{1}{(1-\eta^{-1})T}\right)^{\frac{2(n^2-4)}{n^2}}T^{\frac{4(n-2)}{n^2}}\beta^2\\
&\quad \leq  2^{\frac{2(n^2-4)}{n^2}}e^{n-1}A_0^{\frac{2(n-2)}{n}}\beta^2 \\
&\qquad\times \max\left\{\left(\beta^{\frac{n}{2}} A_0^{\frac{n-2}{2}}(20n)^{\frac{n}{2}}\right)^{\frac{2(n^2-4)}{n^2}},\left(\frac{1}{(1-\eta^{-1})T}\right)^{\frac{2(n^2-4)}{n^2}}\right\}T^{\frac{4(n-2)}{n^2}}.
\end{align*}
Because $T>0$ was arbitrary, this completes the proof. 
\end{proof}


\begin{cor}\label{c4.6} 
Let $(M^n,g(t))$ satisfy the conditions in \eqref{section4assumptions}. Then $T_{max}=\infty$ and 
\[
\sup_{M\times [0,\infty)}|\text{Rm}|(x,t)<\infty.
\]
\end{cor}


\begin{proof}
The proof is identical to {\cite[Corollary 4.6]{ChenEric}}.
\end{proof}


\begin{lem}\label{4.7} 
Let $(M^n,g(t))$ satisfy the conditions in \eqref{section4assumptions}. Then for any $\gamma>0$, there exists $T_\gamma>0$ such that for all $t\geq T_\gamma$, there holds 
\[
\left(\int |\text{Rm}|^{\frac{n}{2}\frac{n}{n-2}} dV_t\right)^{\frac{n-2}{n}}\leq \gamma t^{-1}.
\]
\end{lem}


\begin{proof}
Start by fixing $T>0$. Recall that in Lemma \ref{l4.2}, we proved that for any $\epsilon>0$, there exists $r_\epsilon\gg0$ such that whenever $r\geq r_\epsilon$, there holds 
\begin{multline*}
\int \phi_r^2|\text{Rm}|^{\frac{n}{2}}\,dV_T-\int \phi_r^2|\text{Rm}|^{\frac{n}{2}}\,dV_0 \\
\leq -\frac{1}{10A}\int_0^T  \left(\int \phi_r^{\frac{2n}{n-2}}|\text{Rm}|^{\frac{n}{4}\frac{2n}{n-2}}\,dV_t\right)^{\frac{n-2}{n}}\,dV_t+T\epsilon.
\end{multline*}
Rearranging this in turn gives
\begin{align}
\label{4.7e1}
&\int_0^T\left(\int \phi_r^{\frac{2n}{n-2}}|\text{Rm}|^{\frac{n}{2}\frac{n}{n-2}} dV_t\right)^{\frac{n-2}{n}} \,dt \\
\notag &\quad \leq 10A\left[\int \phi_r^2|\text{Rm}|^{\frac{n}{2}}\,dV_0-\int \phi_r^2|\text{Rm}|^{\frac{n}{2}}\,dV_T\right]+10TA \epsilon\\
\notag &\quad \leq  10A\int \phi_r^2|\text{Rm}|^{\frac{n}{2}}\,dV_0+10TA \epsilon\leq 10A\int|\text{Rm}|^{\frac{n}{2}}\,dV_0+10TA \epsilon.
\end{align}
Now take $r\to \infty$, then $\epsilon\to 0$, then $T\to \infty$ to see that 
\[
\int_0^\infty\left(\int_M |\text{Rm}|^{\frac{n}{2}\frac{n}{n-2}} dV_t\right)^{\frac{n-2}{n}} \,dt\leq 10A\int|\text{Rm}|^{\frac{n}{2}}\,dV_0<\infty.
\]
This combined with the fact that $t\mapsto\int |\text{Rm}|^{\frac{n}{2}\frac{n}{n-2}}\,dt$ is nonincreasing (Lemma \ref{l4.4}) gives
\[
\lim_{t\to \infty}t\left(\int |\text{Rm}|^{\frac{n}{2}\frac{n}{n-2}}\,dV_t\right)^{\frac{n-2}{n}} =0.
\]
\end{proof}


\begin{prop}\label{4.8} 
Let $(M^n,g(t))$ satisfy the conditions in \eqref{section4assumptions}. Then the curvature estimate \ref{maincon3} holds.
\end{prop}


\begin{proof}
The proof is the same as {\cite[Proposition 4.8]{ChenEric}}. 
\end{proof}


At this point, we diverge from the numbering from \cite{ChenEric} to prove the other curvature estimate and the convergence statement. 


\begin{prop}
Let $(M^n,g(t))$ satisfy the conditions in \eqref{section4assumptions}. Then the curvature estimate \ref{maincon2} holds.
\end{prop}


\begin{proof}
By slightly modifying the proof of Proposition \ref{4.5} (see also {\cite[Proposition 4.8]{ChenEric}}) with the only changes being
\[
\tau_k=\frac{T}{2}+(1-\eta^{-k}) \frac{T}{2}, \;\;\text{ and }\; \; \beta=\sup_{t\in [T/2,T]} \||\text{Rm}|\|_{L^q_{g(t)}}, 
\]
we obtain the estimate
\begin{multline}\label{newcurvbound}
\sup_{x\in M}|\text{Rm}|_{g(t)} \leq 
2^{\frac{n^2-4}{n^2}}e^{\frac{1}{2}(n-1)}A_0^{\frac{n-2}{n}}\beta 
 \max\left\{\left(\beta^{\frac{n}{2}} A_0^{\frac{n-2}{2}}(20n)^{\frac{n}{2}} \right)^{\frac{n^2-4}{n^2}},\left(\frac{2}{(1-\eta^{-1})t}\right)^{\frac{n^2-4}{n^2}}\right\}t^{\frac{2(n-2)}{n^2}}.
\end{multline}
Note that Lemma \ref{l4.4} implies that $t\mapsto \int |\text{Rm}|^{q}\,dV_t$ is nonincreasing, so we may take $\beta=\||\text{Rm}|\|_{L^q_{g(T/2)}}$. Now by slightly modifying the proof of Lemma~\ref{4.7} (specifically taking $r\to \infty$ then $\epsilon\to 0$ in inequality \eqref{4.7e1}), we have
\begin{align*}
	\frac{t}{2} \beta^{\frac{n}{2}} &\leq \int_0^{t/2}\left(\int |\text{Rm}|^{\frac{n}{2}\frac{n}{n-2}}\,dV_s\right)^{\frac{n-2}{n}} \,ds\\
					&\leq 10A\int|\text{Rm}|^{\frac{n}{2}}\,dV_0\leq 10 L(n) C_g \left(\frac{\delta_1(n)}{C_g}\right)^{\frac{n}{2}}.
\end{align*}
Then since $A_0=2L(n) C_g$, we can see that 
\[
A_0^{\frac{n-2}{n}}\beta\leq (2L C_g)^{\frac{n-2}{n}}(10 LC_g)^{\frac{2}{n}} \frac{2^{\frac{2}{n}}\delta_1(n)}{C_gt^{\frac{2}{n}}}\leq \frac{1}{4n t^{\frac{2}{n}}}.
\]
Plugging this back into our curvature bound \eqref{newcurvbound} we have
\begin{align*}
\sup_{x\in M}|\text{Rm}|_{g(t)} &\leq 2^{\frac{n^2-4}{n^2}}e^{\frac{1}{2}(n-1)}\frac{1}{4n t^{\frac{2}{n}}}\\
				&\quad \times \max\left\{\left(\frac{(20n)^{\frac{n}{2}}}{(4n)^{\frac{n}{2}} t}\right)^{\frac{n^2-4}{n^2}},\left(\frac{2}{(1-\eta^{-1})t}\right)^{\frac{n^2-4}{n^2}}\right\}t^{\frac{2(n-2)}{n^2}}\\&<
(5e)^{\frac{n}{2}}t^{-\frac{2}{n}-1+\frac{4}{n^2}+\frac{2}{n}-\frac{4}{n^2}}=\frac{(5e)^{\frac{n}{2}}}{t}=\frac{C_0(n)}{t},
\end{align*}
where in the latter inequality we used the fact that 
\[
\frac{2}{(1-\eta^{-1})}\leq 4\leq 5^{\frac{n}{2}}.
\]
\end{proof}


\begin{prop}\label{conv4.10}
Let $(M^n,g(t))$ satisfy the conditions in \eqref{section4assumptions} and furthermore assume that $|\text{Rm}|_g\in L^p$ for some $p<\frac{n}{2}$. Then $g(t)$ converges locally smoothly to a complete flat limit metric $g(\infty)$ on $M^n$ with Euclidean volume growth, and thus $M^n$ is diffeomorphic to $\mathbb{R}^n$. 
\end{prop}


\begin{proof}
In this proof, we will use $C(\cdot, \cdot, \cdots)$ to denote a constant which depends on its arguments, and may differ line-by-line. First note that by interpolation, we may assume that 
\begin{equation}\label{prequirement1}
\begin{cases}
\frac{2}{3}<\frac{p}{2}<\frac{n}{4} & n=3,4\\
\frac{n-2}{4}<\frac{p}{2}<\frac{n}{4} & n\geq 5. 
\end{cases}
\end{equation}
Given this (fixed) value of $p$, write $q=\frac{n}{n-2}p$. Now we note that in establishing \eqref{4.5e5} in the proof of Proposition \ref{4.5}, the only assumptions on $\tilde p, \tilde q$ there were $\tilde p\geq 1$ and $\max\{2,n/2\}\leq \tilde q\leq \frac{n}{2}\frac{n}{n-2}$ so that Lemma 4.4 in turn guaranteed that $t\mapsto\||\text{Rm}|\|_{L^{\tilde q}_{g(t)}}$ is nonincreasing.  Thus for any $\tilde p\geq 1$ and our choice of $q$ above, we have
\begin{align*}
&\frac{d}{dt}\int \phi^2f^{\tilde p} \,dV_t+\int |\nabla (\phi f^{\frac{\tilde p}{2}})|^2\,dV_t\\
&\quad \leq \beta(16\tilde p+n)\left(1-\frac{n}{2q}\right)\left(\frac{\beta n(8\tilde p+n)A_0}{2q}\right)^{\frac{n}{2q-n}}\int \phi^2 f^{\tilde p} \,dV_t +\epsilon\\
&\quad = (\tilde p\beta)^{\frac{2q}{2q-n}}C(n,q,C_g)\int \phi^2 f^{\tilde p} \,dV_t +\epsilon.
\end{align*}

Now we continue on in identical fashion to the proof of Proposition \ref{4.5} (in going from \eqref{4.5e5} to \eqref{4.5e10}) with the only modifications being that we take
\[
\tilde p_0=\frac{q}{2},  \; \; \eta=\nu^{\frac{2q}{2q-n}},  \;\; \tau_k=\frac{T}{2}+(1-\eta^{-k}) \frac{T}{2}, \;\;\text{ and }\; \; \beta=\sup_{t\in [T/2,T]} \||\text{Rm}|\|_{L^q_{g(t)}}.
\]
In light of our choice of $p,q$ and Lemma \ref{l4.4}, we in fact have $\beta=\||\text{Rm}|\|_{L^q_{g(T/2)}}$. Also, by our requirements on $p$, we have that $\tilde p_0>1$. The result after iterating is 
\[
\sup_{x\in M}|\text{Rm}|(x,T)^2 \leq C(n,p, C_g) \left(\beta^{\frac{2q}{2q-n}}+T^{-1}\right)^{\sum_{k=0}^\infty\frac{1}{\tilde p_k}}\Phi_0.
\]
Also, in this case we have
\[
\sum_{k=0}^\infty \frac{1}{\tilde p_k}=\frac{1}{\tilde p_0}\sum_{k=0}^\infty \frac{1}{(1+\frac{2}{n})^k}=\frac{n+2}{q}.
\]
Now recalling the definition of $\Phi_0$ and the fact that $\int_M |\text{Rm}|^q\,dV_t$ is nonincreasing in $t$, we have 
\[
\Phi_0=H(q/2,\tau_0)^{\frac{2}{q}}=\left(\int_{T/2}^T \int_M |\text{Rm}|^{q}\,dV_t \,dt\right)^{\frac{2}{q}}= C(q) T^{\frac{2}{q}}\||\text{Rm}|\|_{L^q_{g(T/2)}}^2.
\]
We need to control the value of $\||\text{Rm}|\|_{L^q_{g(T/2)}}$. For this, we proceed just as in Lemma \ref{l4.2} and \ref{l4.4} with $\alpha=\frac{p}{2}$ (here we use the assumption that $\frac{p}{2}>\frac{2}{3}$ so that we can apply part \ref{4.1b} of Lemma \ref{l4.1}) to obtain

\begin{align*}
-\int|\text{Rm}|^{p}\,dV_0&\leq\int_0^t\left(\int_M |\text{Rm}|^{q} dV_s\right)^{\frac{n-2}{n}} \,ds\left[-\frac{1}{4A}+\left(16\alpha+n\right)\frac{\delta_1(n)}{C_g}\right]\\&\leq 
\int_0^t\left(\int_M |\text{Rm}|^{q} dV_s\right)^{\frac{n-2}{n}} \,ds\left[-\frac{1}{4A}+\frac{5n}{39nA}\right]\\&\leq 
-\frac{1}{10A}\int_0^t\left(\int_M |\text{Rm}|^{q} dV_s\right)^{\frac{n-2}{n}} \,ds\leq -\frac{t}{10A}\||\text{Rm}|\|_{L^q_{g(t)}}^{q\frac{n-2}{n}}.
\end{align*}
Of course, this argument relies on the fact that $t\mapsto \int |\text{Rm}|^p_{g(t)}\,dV_t$ is uniformly bounded on finite time intervals which is true by Proposition \ref{finitenessofn/2} because our assumption \eqref{prequirement1} in particular implies $p>1$. So rearranging the above gives 
\[
\||\text{Rm}|\|_{L^q_{g(t)}}^{q\frac{n-2}{n}}\leq C(n,p,C_g,\||\text{Rm}|\|_{L^p_{g_0}}) \frac{1}{t}
\]
which implies
\[
\Phi_0\leq C(n,p,C_g,\||\text{Rm}|\|_{L^p_{g_0}})  T^{\left(-\frac{n}{n-2}+1\right)\frac{2}{q}}=C(n,p,C_g,\||\text{Rm}|\|_{L^p_{g_0}}) T^{-\frac{4}{q(n-2)}}
\]
and 
\[
\beta^{\frac{2q}{2q-n}}\leq C(n,p,C_g,\||\text{Rm}|\|_{L^p_{g_0}})  T^{-\frac{2n}{(n-2)(2q-n)}}.
\]
Note that 
\[
\frac{2n}{(n-2)(2q-n)} \geq 1 \; \iff \; q\leq \frac{n^2}{2(n-2)}
\]
and the latter is true by our choice of $q$. Therefore for $T$ large, we have
\begin{equation}\label{curvestiminf}
\sup_{x\in M}|\text{Rm}|^2(x,T) \leq C(n,p,C_g,\||\text{Rm}|\|_{L^p_{g_0}})  T^{-\frac{n^2}{q(n-2)}}.
\end{equation}
But $q<\frac{n}{2}\frac{n}{n-2}$ and so we have that $\sup_{x\in M}|\text{Rm}|_{g(t)}\lesssim t^{-(1+\sigma)}$ for some $\sigma>0$. Then by Shi's estimates \cite{Shi2} (see also {\cite[Theorem 4.7]{Li}}), we have that for all $k\geq 0$, there holds
\begin{equation}\label{curvestiminf2}
\sup_{x\in M}|\nabla^k \text{Rm}|_{g(t)}\leq C_k t^{-1-\sigma-k/2}
\end{equation}
for all $t\in (0,\infty)$ and some $C_k(n,p,C_g,\||\text{Rm}|\|_{L^p_{g_0}}) >0$.\\

We now show that the integrability in $t$ of the RHS of \eqref{curvestiminf2} implies that $g(t)$ converges in $C^\infty_{loc}(M)$ as $t\to \infty$ to a complete flat metric $g(\infty)$ on $M$ which is equivalent to $g(0)$.  This is by a standard Ricci flow argument which we now sketch.  In the proceeding paragraphs, the constants $C, C_i, c_i$ can be taken to depend only on $n,p,C_g,\||\text{Rm}|\|_{L^p_{g_0}}$.\\

Let $V\in T_p M$ denote an arbitrary unit vector on $M$ with respect to the metric $g(1)$. Then for any $t \geq 1$ we have 
\begin{align*}
	\log \|V\|_{g(t)}^2&=\log \|V\|_{g(t)}^2-\log \|V\|_{g(1)}^2\\
			   &=\int_1^t \frac{d}{dt}\log \|V\|_{g(t)}^2\,dt= 2\int_1^t \frac{\text{Rc}_{g(t)}(V,V)}{\|V\|^2_{g(t)}}\,dt
\end{align*}
while the integrand on the RHS is bounded in absolute value by $n\sup_M |Rm(x, t)|$ and thus $C/t^{1+\sigma}$ for some $\sigma >0$ by \eqref{curvestiminf2} with $k=0$.  It follows that $\log \|V\|_{g(t)}^2$ converges uniformly as $t\to \infty$ to a positive limit  bounded uniformly  above and away from $0$, where the uniformity is relative to $p$.  From this we conclude that $g(t)$ converges in $L^{\infty}(M)$ as $t\to \infty$ to a metric $g(\infty)$ which is uniformly equivalent to $g(0)$, and thus also complete.\\

We show that in fact  $g(t)\to g(\infty)$ in $C^\infty_{loc}(M)$ and thus
$g(\infty)$ is a smooth flat metric on $M$ equivalent to $g(0)$. Fix a local
coordinate domain $U\subset \subset M$ with coordinates $x^i$, and use
$\tilde{T}$ to denote the local components in $U$ of a global tensor $T$ on
$M$.  Recall that if $\tilde{\Gamma}(t)$ are the Christoffel symbols of
$\tilde{g}(t)$, then $\frac{\partial}{\partial t} \tilde{\Gamma}(t)$ are the
components of a global tensor $g^{-1} * \nabla Rc$ on $M$ for each $t$.
Combining the uniform equivalence of $g(t)$ shown above with
\eqref{curvestiminf2} for $k=1$ we obtain
\begin{equation}\label{EE!!!}
\frac{\partial}{\partial t} \tilde{\Gamma}(t) \leq C_1/t^{1+c_1}.
\end{equation}
Thus $\tilde{\Gamma}(t)$ remains uniformly bounded in $U$ for all $t$.  Now if
we let $\partial$ denote partial derivatives in the coordinates $x^i$, then in
$U$ we have
\begin{equation}\label{EEE!!!!!}
\frac{\partial}{\partial t} \partial \tilde{g}=-2\partial \tilde{Rc}
=-2(\widetilde{\nabla Rc} + \tilde{\Gamma}*\tilde{Rc} )\leq C_2/t^{1+c_2}
\end{equation}
by the boundedness of $\tilde{\Gamma}(t)$ and again the uniform equivalence of
$g(t)$ shown above and \eqref{curvestiminf2} for $k=0, 1$.  From this we
conclude that $\partial \tilde{g}(t)$ converges to a bounded limit on $U$ as
$t\to \infty$.\\

By applying a partial derivative $\partial$ to \eqref{EEE!!!!!}, we may express $\frac{\partial^2}{\partial t^2} \partial \tilde{g}$ in terms of $\widetilde{\nabla^m Rc}$ for $m=0, 1, 2$ and $\frac{\partial^l}{\partial t^l} \tilde{\Gamma}(t)$ for $l=0, 1$ and we can similarly show these terms are bounded by $ C_3/t^{1+c_3}$ from which we conclude that $\partial^2 \tilde{g}(t)$ converges to a bounded limit in $U$ as $t\to \infty$.  Continuing inductively in this way, we can show that  $\partial^m \tilde{g}(t)$ converges to a bounded limit in $U$ as $t\to \infty$ for each $m=3, 4, ...$.  As $U$ was an arbitrary coordinate chart, this establishes that $g(t)$ converges in $C^{\infty}_{loc}(M)$ as $t\to \infty$ to a smooth limit $g(\infty)$ which is equivalent to $g(0)$ (and thus is necessarily complete) and must also be flat by \eqref{curvestiminf2}.\\

We have so far shown that $(M,g(\infty))$ is a complete flat manifold. In particular, its universal cover is $\mathbb{R}^n$ with the Euclidean metric $g_E$ and therefore 
\[
(M,g(\infty)) \cong (\mathbb{R}^n, g_E)/ \Gamma
\]
where $\Gamma$ is a discrete group of isometries which acts freely on $\mathbb{R}^n$. On the other hand, we know that $g(\infty)$ satisfies the Sobolev inequality 
\[
\left(\int_M |u|^{\frac{2n}{n-2}}\,dV_{g(\infty)}\right)^{\frac{n-2}{n}}\leq A_0\int_M |\nabla u|^2\,dV_{g(\infty)}
\]
for all $u\in W^{1,2}(M, g(\infty))$ (this follows because the above inequality holds for all $t>0$ and $g(t)\to g(\infty)$ uniformly in $L^\infty(M)$). So for any fixed domain $\Omega\subset M$ and any $u\in W^{1,2} (\Omega, g(\infty))$, we apply H\"older's inequality to see
\[
\int_\Omega |u|^2 \leq \text{Vol}(\Omega)^{\frac{2}{n}}\left(\int_\Omega |u|^{\frac{2n}{n-2}}\,dV\right)^{\frac{n-2}{n}}\leq A_0\text{Vol}(\Omega)^{\frac{2}{n}}\int_\Omega |\nabla u|^2
\]
and subsequently
\[
\text{Vol}(\Omega)^{\frac{2}{n}}\inf_{u\in W^{1,2}(\Omega, g(\infty))}\frac{\int_\Omega|\nabla u|^2}{\int_\Omega u^2}\geq \frac{1}{A_0}.
\]
Now by {\cite[Proposition 2.4]{Carron}} (or for a proof in English, see {\cite[Lemma 6.1]{Ye}}), we have
\[
\text{Vol}_{g(\infty)}(B_{g(\infty)}(x_0,\rho))\geq \left(\frac{1}{2^{n+2}A_0}\right)^{\frac{n}{2}}\rho^n
\]
for any $\rho>0$. That is to say, $(M,g(\infty))$ has a positive asymptotic volume ratio. Thus if $f(x)=Ax+b$ is an arbitrary element of $\Gamma$ (which must consist of affine maps), the above volume estimate implies $b=0$ and thus $A=1$ since $f$ must be either trivial or fixed point free.  Thus $\Gamma$ must be trivial. This completes the proof. 
\end{proof}

This completes the proof of Theorem \ref{main}. 


\section{Proof of Theorem \ref{main2}}\label{secproofmain2}


In this section we will prove Theorem \ref{main2}, which follows from Theorem \ref{main} and a diagonal argument using Hamilton's Compactness Theorem for Ricci flows.


\begin{proof}[Proof of Theorem \ref{main2}]
Let $(M_i, g_i)$ and $(M, g)$ be as in the Theorem. The hypothesis that 
\[
\sup_{i\in \mathbb{N}}\left(\int_{M_i} |\text{Rm}|_{g_i}^{\frac{n}{2}}\,dV_{g_i}\right)^{\frac{2}{n}}\leq \frac{\delta(n)}{\sup_{i\in \mathbb{N}}C_{g_i}}
\]
implies either that (a) $|\text{Rm}|_{g_i}\equiv 0$ for all $i\in\mathbb{N}$ or else (b) that $\sup_{i\in \mathbb{N}}C_{g_i}<\infty$. In the former case, we must necessarily have that $|\text{Rm}|_g\equiv 0$ since $\phi_i^*g_i\to g$ locally smoothly and the result follows. For the latter case, we know by Theorem \ref{main} that for each $i\in\mathbb{N}$, there exists a long-time complete Ricci flow $g_i(t)$ with $g_i(0)=g_i$. Moreover, the $\{g_i(t)\}$ uniformly satisfy the curvature estimate
\begin{equation}\label{c/tcurvbound11}
\sup_{x\in M_i}|\text{Rm}|_{g_i(t)}\leq \frac{C_0(n)}{t}
\end{equation}
for all $t\in (0,\infty)$. To obtain the flow $g(t)$, we may apply a diagonal argument to the flows $g_i(t)$ along with the Arzel\`a-Ascoli Theorem. We will briefly sketch this argument for the reader's convenience. \\

For any large $R\gg 0$ we may take some $K_R<\infty$ for which 
\[
\sup_{x\in B_g(x_0,2R)}|\text{Rm}|_g\leq K_R.
\]
Since $\phi_i^*g_i\to g$ smoothly on this compact set, we may assume $i$ is large enough to ensure that 
\[
\sup_{x\in B_g(x_0,2R)}|\text{Rm}|_{\phi_i^* g}\leq 2K_R.
\]
Now the uniform curvature bounds along the flow in \eqref{c/tcurvbound11} as well as the uniform time-$0$ curvature bound we just obtained implies (by {\cite[Corollary 3.2]{Chen}}) that 
\[
\sup_{B_{g}(x_0,R)\times [0,\infty)}|\text{Rm}|_{\phi^*_i g_i(t)}\leq 2C(n)K_R.
\] 
Similarly, higher order uniform curvature bounds on compact sets may be obtained by Shi's modified interior estimates {\cite[Theorem 14.16]{Chowetal}}. The existence of the limiting smooth Ricci flow $g(t)$ on $M\times [0,\infty)$ – with the desired estimates –  subsequently follows by a diagonal argument. More precisely, by taking any choice of sequences $R_k,T_k\to \infty$, one may inductively apply the Arzel\`a-Ascoli Theorem to construct (for each $k$) a subsequence $(g_{k_j})_j$, which is a subsequence of the previous $(g_{(k-1)_j})_j$ and is chosen so that $(g_{k_j}(t))_j$ converges in the $C^k$-norm on the compact space-time set $B_{g}(x_0,R_k)\times [0,T_k]$. As a result, the diagonal sequence of flows $(g_{k_k}(t))_k$ converges to a limiting flow $g(t)$ on $M\times [0,\infty)$. \\

Thus, up to a subsequence, we have the convergence 
\[
\phi_{i}^*g_{i} (t)\xrightarrow{C^{\infty}_{loc}(M\times[0, \infty))}{} g(t).
\]
Thus the limiting flow $g(t)$ also satisfies the same curvature estimate \eqref{c/tcurvbound11}. The completeness of the limiting flow $g(t)$ follows from the Shrinking Balls Lemma {\cite[Corollary 3.3]{ST}} because the $C_0/t$ curvature bound \eqref{c/tcurvbound11} descends to the limit flow $g(t)$,  and $g(0)=g$ is assumed to be complete.\\

The convergence statement follows from the proof of Proposition \ref{conv4.10} because the estimates obtained there only depend on $n,p, C_{g_i}$ and $\||\text{Rm}|\|_{L^p_{g_i}}$ which are all assumed to be uniform with respect to $i\in\mathbb{N}$. Thus the estimates \eqref{curvestiminf} and \eqref{curvestiminf2} pass down to the limit $g(t)$. This gives the uniform convergence of $g(t)$ to a flat metric $g(\infty)$ on $M$. The completeness of the limiting metric, and subsequently the topological statement about $M$, are then identical to the proof of Proposition \ref{conv4.10}.  
\end{proof}


\begin{appendix}
\section{Obtaining a bound for $L(n)$}\label{secconstantcalc}


Here we will be concerned with calculating the constants $A$ and $L(n)$ which appear in Theorem \ref{C}. In the proof of that theorem, we establish that for any $t\in [0,T]$ and any $\sigma>0$, there holds the following log-Sobolev inequality:
\begin{align}\label{appendixlogsob}
	\int u^2 \log u^2\,dV_t &\leq 
\sigma \left(\int |\nabla u|^2+\frac{R^+_{t}}{4}\,dV_t\right)-\frac{n}{2}\log \sigma \\
	\notag 			&\quad +\frac{n}{2}\left(\log C_g+\log n-1\right).
\end{align}
As described in Section \ref{sobalongflow}, we can now fix a time $t\in [0,T]$ in this log-Sobolev inequality and use regularity properties of the (fixed metric) heat operator $H=-\Delta_{g(t)}+R_{g(t)}^+/4$ to obtain a Sobolev inequality at time $t$. We will track through the proof presented by Ye {\cite[Appendix C]{Ye}} to calculate the constants in this process. Although Ye restricts to compact manifolds, the proof is verbatim (with the same constants) for bounded domains $\Omega\subset M$ and Dirichlet boundary conditions (see {\cite[Appendix B]{ChenThesis}}). \\

For fixed $t\in [0,T]$, we use \eqref{appendixlogsob} as the hypothesis for {\cite[Theorem 5.3]{Ye}} with 
\[
\beta(\sigma)=-\frac{n}{2}\log \sigma +\frac{n}{2}\left(\log C_g+\log n-1\right).
\]
The output of this theorem is that 
\[
\|e^{-sH}u\|_{L^\infty}\leq e^{\tau(s)}\|u\|_{L^2}
\]
where
\begin{align*}
\tau(s)=\frac{1}{2s}\int_0^s \beta(\sigma)\,d\sigma&=
-\frac{n}{2}\frac{1}{2s}(s\log s -s)+\frac{n}{4}\left(\log C_g+\log n-1\right)\\&=
-\frac{n}{4}\log s+\frac{n}{4}\left(\log C_g+\log n\right)\\
										&=
-\frac{n}{4}\log s+\frac{n}{4}\log (nC_g).
\end{align*}
Then applying {\cite[Theorem 5.4]{Ye}} with $c=(nC_g)^{\frac{n}{4}}$, $\mu=n$, and tracking through the constants in the proof, we obtain
\[
\|u\|_{L_{g(t)}^{\frac{2n}{n-2}}}^2\leq K^2C_{p_0}C_{p_1}\left(\int |\nabla u|^2+\frac{R^+_t}{4}u^2\,dV_t\right).
\]
So we can take $A= K^2C_{p_0}C_{p_1}$ where $K$ and $C_{p_i}$ are as follows. Here $K$ is the constant in the Marcinkiewicz Interpolation Theorem, which we will obtain a bound for later (see Theorem \ref{Mark}). We will take $p_0\in (1,2)$ to be variable (to be optimized later), with $p_1$, $q_0$, $q_1$ defined by
\[
\frac{1}{p_0}+\frac{1}{p_1}=1, \;\; \text{ and } \; \; q_i=\frac{np_i}{n-p_i}.
\] 
The constants $C_{p_i}$ are then given generally as
\[
C_p=\left(2^{\frac{p(2n-p)}{n-p}}c_p^{\frac{p^2}{n-p}}\Gamma\left(\frac{1}{2}\right)^{-p}\right)^{\frac{1}{q}},
\]
where 
\[
c_p=\Gamma\left(\frac{1}{2}\right)^{-1} \left(2^{\frac{n}{2}}c^2\right)^{\frac{1}{p}} \frac{2p}{n-p}. 
\]
Calculating each of the terms involved in $C_p$, we have
\begin{align*}
&\left(2^{\frac{p(2n-p)}{n-p}}2^{\frac{n}{2p}\frac{p^2}{n-p}}\right)^{\frac{n-p}{np}}=2^{\frac{2n-p}{n}}2^{\frac{1}{2}}=2^\frac{5n-2p}{2n}, \;\; \text{ and } \\
&\left(\Gamma\left(\frac{1}{2}\right)^{-\frac{p^2}{n-p}}\Gamma\left(\frac{1}{2}\right)^{-p}\right)^{\frac{n-p}{np}}=\sqrt{\pi}^{-1}.
\end{align*}
Recall $c=(nC_g)^{\frac{n}{4}}$, so that
\[
c^{\frac{2}{p}\frac{p^2}{n-p}\frac{n-p}{np}}=(nC_g)^{\frac{n}{4}\frac{2}{p}\frac{p^2}{n-p}\frac{n-p}{np}}=(nC_g)^{\frac{1}{2}}.
\]
Therefore
\begin{align}
\label{mark1}
A&=K^2 C_{p_0}C_{p_1}\\&= 
K^2\pi^{-1} 2^{\frac{5n-2p_0}{2n}+\frac{5n-2p_1}{2n}}(nC_g)\left(\frac{2p_0}{n-p_0}\right)^{\frac{p_0}{n}}\left(\frac{2p_1}{n-p_1}\right)^{\frac{p_1}{n}}\nonumber\\&=
C_g K^2\left[\frac{n}{\pi}  2^5\left(\frac{p_0}{n-p_0}\right)^{\frac{p_0}{n}}\left(\frac{p_1}{n-p_1}\right)^{\frac{p_1}{n}}\right].\nonumber
\end{align}

Now we need to compute a bound for the constant $K$ which, as noted above, comes from an application of the off-diagonal Markinciwicz Interpolation Theorem (specifically, we require $C_0^{1-t}C_1^t$ dependance in the strong-type estimate where $C_0,C_1$ are the weak-type constants). There are many citations for this particular result, for instance a proof in a more general setting is given in {\cite[Theorem 1.4.19]{Grafakos}} and the upper bound that is obtained after plugging in our constants is something like $K\approx 1000*2^{12} n$.  Through an argument communicated to the authors by Espinoza \cite{Mittens}, one can obtain much sharper bounds by combining the ideas in \cite{Folland} and \cite{Grafakos}.  We duplicate this argument here for the readers' convenience. 


\begin{thm}[\cite{Mittens} Off-Diagonal Markinciwicz Interpolation Theorem with Improved Constant]\label{Mark}
Suppose that $(X,\mu)$ and $(Y,\nu)$ are $\sigma$-finite measure spaces and that $p_0,p_1,q_0,q_1\in [1,\infty]$ such that $p_0<p_1$, $p_0< q_0$, $p_1< q_1$, $q_0< q_1$, and define for $t\in (0,1)$, $p_t$ and $q_t$ by
\begin{equation}\label{pqrelationships}
\frac{1}{p_t}=\frac{1-t}{p_0}+\frac{t}{p_1}, \;\; \frac{1}{q_t}=\frac{1-t}{q_0}+\frac{t}{q_1}.
\end{equation}
If $T: L^{p_0}(X,\mu)+L^{p_1}(X,\mu)\to L_0(\nu)$ is a sub-additive operator which is weak type $(p_i,q_i)$ with constant $C_i$, then $T$ is strong type $(p_t,q_t)$ with constant
\[
C_t=2\left(\frac{q_t}{q_t-q_0}+\frac{q_t}{q_1-q_t}\right)^{\frac{1}{q_t}}C_0^{1-t}C_1^t.
\]
\end{thm}


\begin{proof}[Proof (Duplicated):]
Fix some $t\in (0,1)$ and, for ease of notation, denote $p=p_t$ and $q=q_t$. Let $f\in L^p(X,\mu)$ be normalized so that $\|f\|_p=1$. We will show that 
\begin{equation}\label{markobjec}
\|Tf\|_q\leq 2\left(\frac{q}{q-q_0}+\frac{q}{q_1-q}\right)^{\frac1q}C^{1-t}_0C^t_1.
\end{equation}
Equations \eqref{pqrelationships} imply that
\begin{equation}\label{commonvaluepq}
\frac{p_0(q-q_0)}{q_0(p-p_0)}=\frac{\frac1p\big(\frac{1}{q_0}-\frac{1}{q}\big)}{\frac1q\big(\frac1{p_0}-\frac1p\big)}=
\frac{\frac1p\big(\frac1q-\frac1{q_1}\big)}{\frac1q\big(\frac1p-\frac1{p_1}\big)}=\frac{p_1(q_1-q)}{q_1(p_1-p)}.
\end{equation}
Let us denote by $\sigma$ this common value. For arbitrary $\alpha>0$ and a selected $\epsilon>0$ (to be determined later) we have 
\[
f=f\chi_{\{|f|>\alpha^\sigma\epsilon\}}+ f\chi_{\{|f|\leq \alpha^\sigma\epsilon\}}=:g+h.
\]
Notably, $g\in L^{p_0}(\mu)$ and $h\in L^{p_1}(\mu)$. By the $(p_0,q_0)$ and $(p_1,q_1)$ weak-type estimates, we have
\begin{align*}
\nu(|Tf|>\alpha)&\leq \nu(|Tg|>\alpha/2)+\nu(|Th|>\alpha/2)\\&\leq 
\left(\frac{2C_0}{\alpha}\right)^{q_0}\left(\int_{\{|f|>\alpha^\sigma\epsilon\}}|f|^{p_0}\,d\mu\right)^{\tfrac{q_0}{p_0}}\\
		&\quad +\left(\frac{2C_1}{\alpha}\right)^{q_1}\left(\int_{\{|f|\leq \alpha^\sigma\epsilon\}}|f|^{p_1}\,d\mu\right)^{\frac{q_1}{p_1}}.
\end{align*}
Consequently, there holds
\begin{align*}
\|Tf\|^q_q&=q\int^\infty_0\alpha^{q-1}\nu(|Tf|>\alpha)\,d\alpha\\&\leq 
q(2C_0)^{q_0}\int^\infty_0\alpha^{q-q_0-1}\Big(\int_{\{|f|>\alpha^\sigma\epsilon\}}|f|^{p_0}\,d\mu\Big)^{\tfrac{q_0}{p_0}}\,d\alpha\quad \\&\;\;\;\;+
q(2C_1)^{q_1}\int^\infty_0\alpha^{q-q_1-1}\Big(\int_{\{|f|\leq\alpha^\sigma\epsilon\}}|f|^{p_1}\,d\mu\Big)^{\tfrac{q_1}{p_1}}\,d\alpha.
\end{align*}
The assumptions of $p_j<q_j$ and the $\sigma$-finiteness of the spaces involved allows for us to apply Minkowski's inequality to each of the two integrals on the RHS. This gives
\begin{align*}
&\left(\int^\infty_0\alpha^{q-q_0-1}\left(\int_{\{|f|>\alpha^\sigma\epsilon\}}|f|^{p_0}\,d\mu\right)^{\tfrac{q_0}{p_0}}\,d\alpha\right)^{\frac{p_0}{q_0}} \\
&\quad \leq
\int_X|f|^{p_0}\left(\int^{\frac{|f|^{1/\sigma}}{\epsilon^{1/\sigma}}}_0\alpha^{q-q_0-1}\,d\alpha\right)^{\frac{p_0}{q_0}}\,d\mu\\
&\quad =
\frac{\epsilon^{-\frac{q-q_0}{\sigma}\frac{p_0}{q_0}}}{(q-q_0)^{\frac{p_0}{q_0}}}\int_X|f|^{p_0} |f|^{\frac{q-q_0}{\sigma}\frac{p_0}{q_0}}\,\mu\\
&\quad =
\frac{\epsilon^{-(p-p_0)}}{(q-q_0)^{\frac{p_0}{q_0}}}\int_X|f|^p\,d\mu
\end{align*}
where the last equality follows from \eqref{commonvaluepq}. Similarly
\begin{align*}
\left(\int^\infty_0\alpha^{q-q_1-1}\left(\int_{\{|f|\leq\alpha^\sigma\epsilon\}}|f|^{p_1}\,d\mu\right)^{\tfrac{q_1}{p_1}}\,d\alpha\right)^{\frac{p_1}{q_1}} 
\leq\frac{\epsilon^{p_1-p}}{(q_1-q)^{\frac{p_1}{q_1}}}\int_X|f|^p\,d\mu.
\end{align*}
Putting things together, we obtain that  for all $f\in L^p(\mu)$ with $\|f\|_p=1$, there holds
\[
\|Tf\|^q_q\leq q\left(\frac{(2C_0)^{q_0}\epsilon^{-\frac{q_0}{p_0}(p-p_0)}}{q-q_0} + \frac{(2C_1)^{q_1}\epsilon^{\frac{q_1}{p_1}(p_1-p)}}{q_1-q}\right).
\]
Now we choose $\epsilon>0$ so that 
\begin{equation}\label{defnofdeltamark}
(2C_0)^{q_0}\epsilon^{-\frac{q_0}{p_0}(p-p_0)}= (2C_1)^{q_1}\epsilon^{\frac{q_1}{p_1}(p_1-p)}
\end{equation}
which subsequently yields
\begin{equation}\label{almostlaststepmark}
\|Tf\|^q_q\leq \left(\frac{q}{q-q_0}+\frac{q}{q_1-q}\right)(2C_1)^{q_1}\epsilon^{\frac{q_1}{p_1}(p_1-p)}.
\end{equation}
Equation \eqref{defnofdeltamark}, along with \eqref{commonvaluepq},  implies that
\[
\frac{(2C_0)^{q_0}}{(2C_1)^{q_1}}=\epsilon^{\frac{q_1}{p_1}(p_1-p)+\frac{q_0}{p_0}(p-p_0)}=\epsilon^{\frac{q_1-q}{\sigma}+\frac{q-q_0}{\sigma}}=\epsilon^{\frac{q_1-q_0}{\sigma}}.
\]
Using this value of $\epsilon$, our definition of $\sigma$ from \eqref{commonvaluepq}, and our assumption on $p_i, q_i$ as in \eqref{pqrelationships} yields
\begin{align*}
	(2C_1)^{q_1}\epsilon^{\frac{q_1}{p_1}(p_1-p)}&=(2C_1)^{q_1}\left(\frac{(2C_0)^{q_0}}{(2C_1)^{q_1}}\right)^{\sigma\frac{q_1(p_1-p)}{p_1(q_1-q_0)}}=
(2C_1)^{q_1}\left(\frac{(2C_0)^{q_0}}{(2C_1)^{q_1}}\right)^{\frac{q_1-q}{q_1-q_0}}\\
						     &=
(2C_0)^{q\frac{\frac{1}{q}-\frac{1}{q_1}}{\frac{1}{q_0}-\frac{1}{q_1}}}(2C_1)^{q_1-q_1\frac{q_1-q}{q_1-q_0}}=
(2C_0)^{q(1-t)}(2C_1)^{qq_1\left(\frac{1}{q}-\frac{(1-t)}{q_0}\right)}\\
						     &=
2^qC_0^{q(1-t)}C_1^{qt}.
\end{align*}
Putting this back into \eqref{almostlaststepmark} and taking $q$'th roots on both sides gives \eqref{markobjec} and completes the proof. 
\end{proof}


In our setting, we are applying Theorem \ref{Mark} with our given $p_i$, $q_i$ and with $t=\frac{1}{2}$. Thus we have 
\[
K= 2\left(\frac{q}{q-q_0}+\frac{q}{q_1-q}\right)^{\frac{1}{q}}.
\]
Plugging this back into \eqref{mark1}, we have
\[
A\leq C_g \left[2\left(\frac{q}{q-q_0}+\frac{q}{q_1-q}\right)^{\frac{1}{q}}\right]^2\left[\frac{n}{\pi}  2^5\left(\frac{p_0}{n-p_0}\right)^{\frac{p_0}{n}}\left(\frac{p_1}{n-p_1}\right)^{\frac{p_1}{n}}\right].
\]
Recall that we did all this with a variable $p_0\in (1,2)$. We can optimize this bound by finding its infimum over this interval. Thus we take 
\[
L(n)=\inf_{p_0\in (1,2)} \left\{n\frac{2^5}{\pi}\left[2\left(\frac{q}{q-q_0}+\frac{q}{q_1-q}\right)^{\frac{1}{q}}\right]^2\left(\frac{p_0}{n-p_0}\right)^{\frac{p_0}{n}}\left(\frac{p_1}{n-p_1}\right)^{\frac{p_1}{n}}\right\}.
\]
For any $n\geq 3$, this infimum is achieved for some $p_0\in (1,2)$. We can calculate numerical bounds for $L(n)$ using a computer:
\[
L(n)<\begin{cases}
646 & n=3\\
434 & n=4\\
430 & n=5\\
78n & n\geq 6.
\end{cases}.
\]


\begin{remark}
In {\cite[Theorem 5.4]{Ye}}, Ye gives precise values of $p_0,p_1$ of  
\[
p_0=\frac{n+2}{n} \;\;\text{ and } \; \;  p_1=\frac{n+2}{2}.
\]
\end{remark}

\end{appendix}

\bibliography{longtime}
\end{document}